\documentclass[]{amsart}

\usepackage{amssymb}
\usepackage{amsfonts}
\usepackage{amsmath}
\usepackage{amsthm}
\usepackage{graphicx}
\usepackage{mathrsfs}
\usepackage{dsfont}
\usepackage{amscd}
\usepackage{multirow}
\usepackage[all]{xy}
\normalfont
\usepackage[T1]{fontenc}
\usepackage{calligra}
\usepackage{verbatim}
\usepackage[usenames,dvipsnames]{color}
\usepackage[colorlinks=true,linkcolor=Blue,citecolor=Violet]{hyperref}
\usepackage{enumerate}

\newcommand{\N}{{\mathds{N}}}
\newcommand{\Z}{{\mathds{Z}}}

\newcommand{\R}{{\mathds{R}}}
\newcommand{\C}{{\mathds{C}}}
\newcommand{\T}{{\mathds{T}}}
\newcommand{\U}{{\mathds{U}}}
\newcommand{\D}{{\mathfrak{D}}}
\newcommand{\A}{{\mathfrak{A}}}
\newcommand{\B}{{\mathfrak{B}}}

\newcommand{\bigslant}[2]{{\raisebox{.2em}{$#1$}\left/\raisebox{-.2em}{$#2$}\right.}}
\newcommand{\Nbar}{\overline{\N}}
\newcommand{\Lip}{{\mathsf{L}}}

\newcommand{\propinquity}[1]{{\mathsf{\Lambda}_{#1}}}
\newcommand{\dpropinquity}[1]{{\mathsf{\Lambda}^\ast_{#1}}}
\newcommand{\covpropinquity}[1]{{\mathsf{\Lambda}^{\mathrm{cov}}_{#1}}}

\newcommand{\Kantorovich}[1]{{\mathsf{mk}_{#1}}}

\newcommand{\Haus}[1]{{\mathsf{Haus}_{#1}}}

\newcommand{\StateSpace}{{\mathscr{S}}}

\newcommand{\MongeKant}{{Mon\-ge-Kan\-to\-ro\-vich metric}}

\newcommand{\Lqcms}{{\JLL} quantum compact metric space}
\newcommand{\Qqcms}[1]{{$#1$}--\gQqcms}
\newcommand{\gQqcms}{quasi-Leibniz quantum compact metric space}

\newcommand{\qcms}{quantum compact metric space}

\newcommand{\unit}{1}

\newcommand{\sa}[1]{{\mathfrak{sa}\left({#1}\right)}}

\newcommand{\UIso}[4]{{\mathsf{UIso}_{#1}\left({#2}\rightarrow{#3}\middle\vert{#4}\right)}}

\newcommand{\inner}[3]{{\left<{#1},{#2}\right>_{#3}}}
\newcommand{\JLL}{Lei\-bniz}

\newcommand{\bridge}[1]{#1} 
\newcommand{\dom}[1]{{\operatorname*{dom}\left({#1}\right)}}
\newcommand{\codom}[1]{{\operatorname*{codom}\left({#1}\right)}}

\newcommand{\norm}[2]{{\left\|{#1}\right\|_{#2}}}
\newcommand{\tunnelset}[4]{{\text{\calligra Tunnels}\,\left[{#1}\stackrel{#3}{\longrightarrow}{#2}\middle\vert {#4} \right]}}
\newcommand{\tunnelsetltd}[3]{{\text{\calligra Tunnels}\,\left[{#1}\stackrel{#3}{\longrightarrow}{#2} \right]}}

\newcommand{\bridgereach}[2]{{\varrho\left(\bridge{#1}\middle\vert {#2}\right)}}

\newcommand{\bridgelength}[2]{{\lambda\left(\bridge{#1}\middle\vert{#2}\right)}}

\newcommand{\bridgenorm}[2]{{\mathsf{bn}_{ \bridge{#1}  }\left({#2}\right)}}

\newcommand{\targetsettunnel}[3]{{\mathfrak{t}_{#1}\left({#2}\middle\vert{#3}\right)}}

\newcommand{\worknote}[1]{} 

\newcommand{\tunnelreach}[2]{{\rho\left({#1}\middle\vert{#2}\right)}}

\newcommand{\tunnelmagnitude}[2]{{\mu\left({#1}\middle\vert{#2}\right)}}

\newcommand{\tunnelextent}[1]{{\chi\left({#1}\right)}}

\newcommand{\alg}[1]{{\mathfrak{#1}}}

\newcommand{\dil}[1]{{\mathrm{dil}\left({#1}\right)}}

\theoremstyle{plain}
\newtheorem{theorem}{Theorem}[section]

\newtheorem{corollary}[theorem]{Corollary}

\newtheorem{lemma}[theorem]{Lemma}
\newtheorem{proposition}[theorem]{Proposition}

\newtheorem{theorem-definition}[theorem]{Theorem-Definition}

\theoremstyle{definition}
\newtheorem{definition}[theorem]{Definition}

\newtheorem{convention}[theorem]{Convention}
\newtheorem{hypothesis}[theorem]{Hypothesis}

\theoremstyle{remark}

\newtheorem{remark}[theorem]{Remark}

\newtheorem{notation}[theorem]{Notation}

\renewcommand{\geq}{\geqslant}
\renewcommand{\leq}{\leqslant}

\numberwithin{equation}{section}

\allowdisplaybreaks[4]

\begin{document}

\title{The Covariant Gromov-Hausdorff Propinquity}
\author{Fr\'{e}d\'{e}ric Latr\'{e}moli\`{e}re}
\email{frederic@math.du.edu}
\urladdr{http://www.math.du.edu/\symbol{126}frederic}
\address{Department of Mathematics \\ University of Denver \\ Denver CO 80208}

\date{\today}
\thanks{This work is part of the project supported by the grant H2020-MSCA-RISE-2015-691246-QUANTUM DYNAMICS and grant \#3542/H2020/2016/2 of the Polish Ministry of Science and Higher Education.}

\begin{abstract}
  We extend the Gromov-Hausdorff propinquity to a metric on Lipschitz dynamical systems, which are given by strongly continuous actions of proper monoids on quantum compact metric spaces via Lipschitz morphisms. We prove that our resulting metric is zero between two Lipschitz dynamical systems if and only if there exists an equivariant full quantum isometry between. We apply our work to convergence of the dual actions on fuzzy tori to the dual actions on quantum tori. Our framework is general enough to also allow for the study of the convergence of continuous semigroups of positive linear maps and other actions of proper monoids.
\end{abstract}
\maketitle



\section{Introduction}

The Gromov-Hausdorff propinquity \cite{Latremoliere13,Latremoliere13b,Latremoliere14,Latremoliere15} is a complete metric, up to full quantum isometry, on classes of quantum compact metric spaces, which induces the same topology as the Gromov-Hausdorff distance \cite{Gromov81} on the class of classical compact spaces. The propinquity is designed as a new tool to discuss approximations of quantum spaces, with the dual purpose to provide a new framework for noncommutative metric geometry and to help the construction of physical theories over quantum spaces. In this paper, we introduce a covariant extension of the propinquity, designed to capture symmetries and additional geometric properties of quantum spaces. There are two main prospective applications for including actions of groups or monoids on quantum spaces in our noncommutative metric geometry framework. First, symmetries encoded by group actions are a crucial ingredient for physical theories, and thus, approximating a physical theory would naturally include approximation of its symmetries. In some instances, symmetries may actually include all the geometry of the underlying space, as for instance for the quantum torus. Second, noncommutative heat semigroups, usually modeled as semigroups of completely positive maps, can be used to define noncommutative geometries \cite{Sauvageot03}, or to model time-evolutions of physical systems. Thus, this paper is part of our program to extend the idea of convergence of spaces, as introduced by Edwards \cite{Edwards75} for compact metric spaces and Gromov \cite{Gromov81} for proper metric spaces, to noncommutative spaces, with our propinquity \cite{Latremoliere13,Latremoliere13b}, and then to their associated structures, such as modules \cite{Latremoliere16c,Latremoliere17a,Latremoliere18a}, and now symmetries and dynamics.

We showed in \cite{Latremoliere17c} that convergence for the Gromov-Hausdorff propinquity preserves symmetries, in a general sense. It is therefore natural to define a covariant propinquity, which now includes the convergence of symmetries themselves, by quantifying how far any two Lipschitz dynamical systems are, where Lipschitz dynamical systems are given by a quantum metric space and an action of a proper metric monoid on that space via Lipschitz morphisms. 

Quantum metric spaces' definition has evolved with time \cite{Connes89,Rieffel98a,Rieffel00,Latremoliere13,Latremoliere05b,Latremoliere12b,Latremoliere15}, in order to suit the development of noncommutative metric geometry, and in particular, the construction of an appropriate noncommutative analogue of the Gromov-Hausdorff distance \cite{Rieffel00,Li06,Rieffel10c}. We will use the following definition for the work in this paper, which is the definition for our entire framework based on the Gromov-Hausdorff propinquity \cite{Latremoliere15b}, which appears to be a good candidate for the sought-after noncommutative Gromov-Hausdorff distance. Our notion of a {\qcms} involves a choice of a particular ``quasi-Leibniz property'' as explained in the definitions below.

\begin{notation}
  Throughout this paper, for any unital C*-algebra $\A$, the norm of $\A$ is denoted by $\norm{\cdot}{\A}$, the space of self-adjoint elements in $\A$ is denoted by $\sa{\A}$, the unit of $\A$ is denoted by $\unit_\A$ and the state space of $\A$ is denoted by $\StateSpace(\A)$. We also adopt the convention that if a seminorm $\Lip$ is defined on some dense subspace of $\sa{\A}$ and $a\in\sa{\A}$ is not in the domain of $\Lip$, then $\Lip(a) = \infty$.
\end{notation}

\begin{definition}\label{qcms-def}
  A \emph{\qcms} $(\A,\Lip)$ is an ordered pair of a unital C*-algebra $\A$ and a seminorm $\Lip$, called an \emph{L-seminorm}, defined on a dense Jordan-Lie subalgebra $\dom{\Lip}$ of $\sa{\A}$, such that:
  \begin{enumerate}
    \item $\{ a \in \sa{\A} : \Lip(a) = 0 \} = \R\unit_\A$,
    \item the \emph{\MongeKant} $\Kantorovich{\Lip}$ defined for any two states $\varphi, \psi \in \StateSpace(\A)$ by:
      \begin{equation*}
        \Kantorovich{\Lip}(\varphi, \psi) = \sup\left\{ |\varphi(a) - \psi(a)| : a\in \dom{\Lip}, \Lip(a) \leq 1 \right\}
      \end{equation*}
      metrizes the weak* topology restricted to $\StateSpace(\A)$,
    \item $\Lip$ satisfies the $F$-quasi-Leibniz inequality, i.e. for all $a,b \in \dom{\Lip}$:
      \begin{equation*}
        \max\left\{ \Lip\left(\frac{a b + b a}{2}\right), \Lip\left(\frac{a b - b a}{2i}\right) \right\} \leq F(\norm{a}{\A},\norm{b}{\A},\Lip(a),\Lip(b))\text{,}
      \end{equation*}
      for some \emph{permissible} function $F$, i.e. a function $F : [0,\infty)^4\rightarrow [0,\infty)$, increasing when $[0,\infty)^4$ is endowed with the product order, and such that for all $x,y,l_x,l_y \geq 0$ we have $F(x,y,l_x,l_y) \geq x l_y + y l_x$,
    \item $\Lip$ is lower semi-continuous with respect to $\norm{\cdot}{\A}$.
  \end{enumerate}
We say that $(\A,\Lip)$ is \emph{Leibniz} when $F$ can be chosen to be $F:x,y,l_x,l_y \mapsto x l_y + y l_x$. More generally, if $\Lip$ satisfies the $F$-quasi-Leibniz inequality for some $F$ then $(\A,\Lip)$ is called a {\Qqcms{F}}.
\end{definition}
We note that a generalization of Definition (\ref{qcms-def}) to the locally compact setting appears in some of our research \cite{Latremoliere05b,Latremoliere12b}.

Quantum compact metric spaces form a category for the appropriate choices of morphisms. We refer to \cite{Latremoliere16b} for some observations on the definition of Lipschitz morphisms and some of their applications. The definition of quantum isometry relies on a key observation of Rieffel in \cite{Rieffel00}.

\begin{definition}
  Let $(\A,\Lip_\A)$ and  $(\B,\Lip_\B)$ be {\qcms s}.
  \begin{itemize}
  \item A positive unital linear map $\pi : \A \rightarrow \B$ is \emph{Lipschitz} when there exists $k\geq 0$ such that $\Lip_\B\circ\pi\leq k \Lip_\A$.
  \item A \emph{Lipschitz morphism} $\pi : \A \rightarrow \B$ is a unital *-endomorphism from $\A$ to $\B$ when there exists $k\geq 0$ such that $\Lip_\B\circ\pi\leq k \Lip_\A$.
  \item A \emph{quantum isometry} $\pi : (\A,\Lip_\A) \rightarrow (\B,\Lip_\B)$ is a *-epimorphism from $\A$ onto $\B$ such that for all $b \in \dom{\Lip_\A}$:
    \begin{equation*}
      \Lip_\B(b) = \inf\left\{ \Lip_\A(a) : \pi(a) = b \right\} \text{.}
    \end{equation*}
  \item A \emph{full quantum isometry} $\pi : (\A,\Lip_\A) \rightarrow (\B,\Lip_\B)$ is a *-isomorphism from $\A$ onto $\B$ such that $\Lip_\B\circ\pi = \Lip_\A$.
  \end{itemize}
\end{definition}

Among all the Lipschitz morphisms between {\qcms s}, full quantum isometries provide the appropriate notion of isomorphism for our purpose. 

The \emph{Gromov-Hausdorff propinquity} is a complete metric on the class of {\Qqcms{F}s} up to full quantum isometry, for any choice of a continuous permissible $F$. We refer to \cite{Latremoliere13b,Latremoliere14,Latremoliere15b} for the motivation, background, and main results on the Gromov-Hausdorff propinquity. The Gromov-Hausdorff propinquity is constructed using the dual analogue of isometric embeddings for {\qcms s}, which we call tunnels:
\begin{definition}
  Let $F$ be a permissible function, and let $(\A_1,\Lip_1)$ and $(\A_2,\Lip_2)$ be two {\Qqcms{F}s}. An \emph{$F$-tunnel} $\tau = (\D,\Lip,\pi_1,\pi_2)$ from $(\A_1,\Lip_1)$ to $(\A_2,\Lip_2)$ is a {\Qqcms{F}} $(\D,\Lip)$ and two quantum isometries $\pi_1 : (\D,\Lip)\twoheadrightarrow (\A_1,\Lip_1)$ and $\pi_2 : (\D,\Lip)\twoheadrightarrow(\A_2,\Lip_2)$. The \emph{domain} $\dom{\tau}$ of $\tau$ is $(\A_1,\Lip_1)$ while the codomain $\codom{\tau}$ of $\tau$ is $(\A_2,\Lip_2)$.
\end{definition}
In particular, tunnels give rise to isometric embeddings of the state spaces, though the isometries are of a very special kind, as dual maps to *-monomorphisms,  as illustrated in Figure (\ref{tunnel-fig}). Fixing a permissible function $F$ and two {\Qqcms{F}s} $(\A,\Lip_\A)$ and $(\B,\Lip_\B)$, the set of all $F$-tunnels from $(\A,\Lip_\A)$ to $(\B,\Lip_\B)$ is denoted by:
\begin{equation*}
  \tunnelsetltd{(\A,\Lip_\A)}{(\B,\Lip_\B)}{F} \text{.}
\end{equation*}

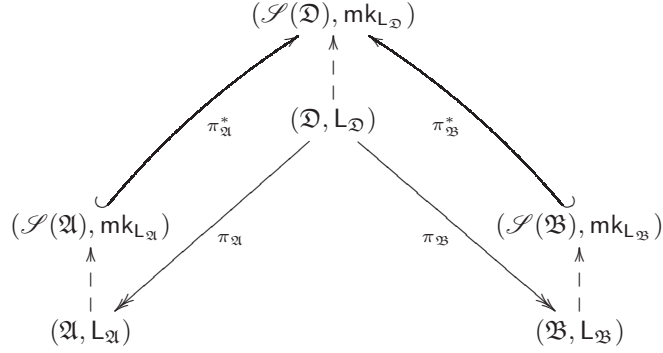
\begin{figure}[t]\label{tunnel-fig}
\begin{equation*}
  \xymatrix{
    & (\StateSpace(\D),\Kantorovich{\Lip_\D})  & \\
    & (\D,\Lip_\D) \ar@{>>}[ldd]^{\pi_\A} \ar@{>>}[rdd]_{\pi_\B} \ar@{-->}[u] & \\
    (\StateSpace(\A),\Kantorovich{\Lip_\A}) \ar@/^/@{^{(}->}[ruu]_{\pi_\A^\ast} & & (\StateSpace(\B),\Kantorovich{\Lip_\B}) \ar@/_/@{_{(}->}[luu]^{\pi_\B^\ast} \\
    (\A,\Lip_\A) \ar@{-->}[u] & & (\B,\Lip_\B) \ar@{-->}[u]
    }
\end{equation*}
\caption{A tunnel and the dual isometric embeddings of state spaces}
\begin{tabular}{lp{4cm}}
  $\hookrightarrow$ & isometry \\
  $\twoheadrightarrow$ & quantum isometry \\
  dotted arrows & duality relations \\
  $\pi^\ast : \varphi \mapsto \varphi\circ\pi$ & dual map\\
  $(\A,\Lip_\A)$, $(\B,\Lip_\B)$, $(\D,\Lip_\D)$ & {\Qqcms{F}s}
\end{tabular}
\end{figure}
There is a natural quantity associated with any tunnels which, in essence, measures how far apart the domain and codomain of a tunnel are for this particular choice of embedding.
\begin{definition}
  Let $(\A_1,\Lip_1)$ and $(\A_2,\Lip_2)$ be two {\qcms s}. The \emph{extent} $\tunnelextent{\tau}$ of a tunnel $\tau$ from $(\A_1,\Lip_1)$ to $(\A_2,\Lip_2)$ is the nonnegative number:
  \begin{equation*}
    \tunnelextent{\tau} = \max_{j \in \{1,2\}} \Haus{\Kantorovich{\Lip}}\left(\left\{ \varphi\circ\pi_j : \varphi \in \StateSpace(\A_j) \right\}, \StateSpace(\D)\right)\text{.}
  \end{equation*}
\end{definition}
We note that the extent of a tunnel is always finite. The propinquity is thus defined as follows:
\begin{definition}
  Let $F$ be a permissible function. For any two {\Qqcms{F}s} $(\A,\Lip_\A)$ and $(\B,\Lip_\B)$, the \emph{dual Gromov-Hausdorff $F$-propinquity} $\dpropinquity{F}((\A,\Lip_\A),(\B,\Lip_\B))$ is the nonnegative number:
  \begin{equation*}
    \dpropinquity{F}\left((\A,\Lip_\A),(\B,\Lip_\B)\right) = \inf\left\{ \tunnelextent{\tau} : \tau \in \tunnelsetltd{(\A,\Lip_\A)}{(\B,\Lip_\B)}{F} \right\}\text{.}
  \end{equation*}
\end{definition}
The propinquity enjoys the properties which a noncommutative analogue of the Gromov-Hausdorff distance ought to possess, though it was not a trivial task to unearth this definition.
\begin{theorem}
  Let $F$ be a permissible function. The $F$-propinquity $\dpropinquity{F}$ is a complete metric up to full quantum isometry on the class of {\Qqcms{F}s}. Moreover, the class map which associates, to any compact metric space $(X,d)$, its canonical {\Lqcms} $(C(X),\Lip_d)$ where $\Lip_d$ is the Lipschitz seminorm,  is an homeomorphism onto its range, when its domain is endowed with the Gromov-Hausdorff distance topology and its codomain is endowed with the topology induced by the dual propinquity.
\end{theorem}
Examples of interesting convergences for the propinquity include fuzzy tori approximations of quantum tori \cite{Latremoliere13c}, continuity for certain perturbations of quantum tori \cite{Latremoliere15c}, unital AF algebras with faithful tracial states \cite{Latremoliere15d}, continuity for noncommutative solenoids \cite{Latremoliere16}, and Rieffel's work on approximations of spheres by full matrix algebras \cite{Rieffel15}, among other examples. Moreover, we prove \cite{Latremoliere15b} an analogue of Gromov's compactness theorem.

We make two remarks regarding the construction of the propinquity which will play a role in this paper. First, it is possible to restrict which tunnels between any two {\qcms s} are used in the definition of the propinquity, besides fixing a quasi-Leibniz inequality via the choice of a permissible function. This may prove helpful, for instance, to impose additional conditions such as the strong Leibniz condition of Rieffel \cite{Rieffel10c}. We refer to \cite{Latremoliere14} for the conditions one needs to impose on whichever choice of tunnels for the construction of the propinquity to give a metric up to full quantum isometry. We will encounter this feature  of our construction in the present paper as well.

The second remark concerns the matter of actually constructing tunnels. The key source for tunnels in our work are objects which we call \emph{bridges} \cite{Latremoliere13}, after a more general notion from Rieffel \cite{Rieffel00}. We recall from \cite{Latremoliere13} the definition of bridges in Definition (\ref{bridge-def}) in this paper, followed by the definition of the length of bridges between {\qcms s}. An important observation from \cite{Latremoliere13c} is that bridges provide natural tunnels whose extend is bounded by twice the length of the original bridge, and in fact, in all the examples of convergence for the propinquity cited above, bridges are used. It is actually possible to work only with tunnels arising from bridges in that way --- this was in fact how our work begun in \cite{Latremoliere13} --- though, in particular, the dual completeness of the dual propinquity arises from allowing for the more general framework we have described above.

We prove in this paper that we can construct a distance on the class of Lipschitz dynamical systems which is null if and only if there exists an \emph{equivariant} full quantum isometry between Lipschitz dynamical systems, i.e. a full quantum isometry between the quantum spaces and an isometric isomorphism between the monoids, such that the actions are intertwined by these maps. We illustrate an application of the covariant propinquity by proving that quantum tori with their dual actions are limits of fuzzy tori with their own dual actions, strengthening the results in \cite{Latremoliere13c} by now including convergence of the symmetries. The natural question of completeness of the covariant propinquity is addressed in the companion paper \cite{Latremoliere18c}.

We wish to apply the covariant propinquity to various situations in future works. First, there are other examples of approximations of quantum spaces whose geometry arises from their symmetries, such as convergence of full matrix algebras to the $2$-sphere, where the geometry is given by the action of $SU(2)$ \cite{Rieffel01,Rieffel10c,Rieffel15}. Second, our covariant propinquity provides a distance on dynamical systems, and thus we wish to study the function which, to a C*-dynamical system, associates its C*-crossed-product, from the covariant propinquity to the propinquity. This works will require the study of the metric geometry of C*-crossed-products. Note that convergence of C*-crossed-product together with their dual actions, in the sense of the covariant propinquity, ought to be tightly connected with the properties of the underlying dynamical system. Another direction, for which the theory in this paper is designed at an appropriate level of generality, is the study of convergence of quantum spaces together with actions of monoids via positive unital maps, with a sight in particular on convergence of continuous semigroups of completely positive maps. This opens a new route for the study of convergence of geometries, for such geometries that arise via such semigroups via their associated Dirichlet forms and differential calculi.

We begin our paper with the notion of convergence of proper monoids, i.e. a monoid with a left invariant metric whose closed balls are compact. Our Definition (\ref{near-iso-def}) is the natural way to define an approximate isometric isomorphism in the spirit of Gromov-Hausdorff convergence. However, it is not well-suited to define a metric between monoids, as difficulties arise when trying to prove the triangle inequality. Instead, we introduce Definition (\ref{almost-iso-def}), which we prove is, in some sense, equivalent to Definition (\ref{near-iso-def}), though it is easier to describe, and does lead us to the definition of a Gromov-Hausdorff type metric between proper monoids, up to isometric isomorphism of monoids. We prove that our metric between proper monoids dominates the pointed Gromov-Hausdorff distance between proper metric spaces. 

We then turn to the construction of the covariant propinquity. Our construction merges our notion of tunnel form \cite{Latremoliere13b} and our notion of almost isometry. We prove that the resulting covariant tunnels can be (almost) composed as in \cite{Latremoliere14}, and that indeed, our new distance is up to \emph{equivariant} full quantum isometry. We conclude our paper with an application to quantum and fuzzy tori, where we prove convergence of the dual actions of closed subgroups of tori on their quantum counterpart for the covariant propinquity. This section outlines how to extend the notion of bridges \cite{Latremoliere13} to our covariant framework.

\section{A Gromov-Hausdorff distance for proper monoids}

We define a covariant Gromov-Hausdorff distance on the class of proper monoids endowed with a left invariant metric. A \emph{monoid} is an associative magma with a neutral element, referred to as the identity element of the monoid. On the other hand, a \emph{proper metric space} is a metric space whose closed balls are all compact.

\begin{definition}
  A \emph{metric monoid} $(G,\delta)$ (resp. group) is a monoid (resp. a group) $G$ and a left invariant metric $\delta$ on $G$ for which the multiplication is continuous (resp. the multiplication and the inverse function are continuous).

The metric monoid (resp. group) is \emph{proper} when all its closed balls are compact.
\end{definition}

\begin{definition}
  A \emph{(metric monoid) morphism} $\pi : G \rightarrow H$ is a map such that:
  \begin{itemize}
    \item $\pi$ maps the identity element of $G$ to the identity element of $H$,
    \item $\forall g,h \in G \quad \pi(g h) = \pi(g) \pi(h)$,
    \item $\pi$ is continuous.
  \end{itemize}
\end{definition}

\begin{notation}
  For a metric space $(X,\delta)$, $x\in X$ and $r\geq 0$, the closed ball in $(X,\delta)$ centered at $x$, of radius $r$, is denoted as $X_\delta[x,r]$, or simply $X[x,r]$.
  If $(G,\delta)$ is a metric monoid with identity element $e \in G$, and if $r \geq 0$, then $G[e,r]$ is denoted as $G[r]$. 
\end{notation}

\begin{remark}
  A proper metric space is always complete and separable.
\end{remark}

We define our distance between two proper metric monoids $(G_1,\delta_1)$ and $(G_2,\delta_2)$ by measuring how far a given pair of maps $\varsigma_1:G_1\rightarrow G_2$ and $\varsigma_2 : G_2\rightarrow G_1$ is from being an isometric isomorphism and its inverse. There are at least two ways to do so. The following definition will serve this purpose well for us.

\begin{definition}\label{almost-iso-def}
Let $(G_1,\delta_1)$ and $(G_2,\delta_2)$ be two metric monoids with respective identity elements $e_1$ and $e_2$. An \emph{$r$-local $\varepsilon$-almost isometric isomorphism} $(\varsigma_1,\varsigma_2)$, for $\varepsilon \geq 0$ and $r \geq 0$,  is an ordered pair of maps $\varsigma_1 : G_1[r] \rightarrow G_2$ and $\varsigma_2 : G_2[r] \rightarrow G_1$ such that for all $\{j,k\} = \{1,2\}$:
\begin{equation*}
\forall g,g' \in G_j[r] \quad \forall h \in G_k[r] \quad \left| \delta_k(\varsigma_j(g)\varsigma_j(g'),h) - \delta_j(gg',\varsigma_k(h))\right| \leq \varepsilon\text{,}
\end{equation*}
and
\begin{equation*}
\varsigma_j(e_j) = e_k \text{.}
\end{equation*}
 The set of all $r$-local $\varepsilon$-almost isometric isomorphism is denoted by:
\begin{equation*}
  \UIso{\varepsilon}{(G_1,\delta_1)}{(G_2,\delta_2)}{r} \text{.}
\end{equation*}
\end{definition}

\begin{convention}
  Write $f_{\big|D}$ for the restriction of a function $f$ to some subset $D$ of its domain. Let $\varsigma : D_1 \subseteq G_1 \rightarrow G_2$ and $\varkappa : D_2 \subseteq G_2 \rightarrow G_1$ with $G_j[r] \subseteq D_j$ for some $r \geq 0$ and $j \in \{1,2\}$. For any $\varepsilon \geq 0$, we will simply write $(\varsigma,\varkappa) \in \UIso{\varepsilon}{(G_1,\delta_1)}{(G_2,\delta_2)}{r}$ to mean:
  \begin{equation*}
    \left(\varsigma_{\big|G_1[r]},\varkappa_{\big| G_2[r]}\right) \in \UIso{\varepsilon}{(G_1,\delta_1)}{(G_2,\delta_2)}{r} \text{.}
  \end{equation*}
\end{convention}

Our covariant Gromov-Hausdorff distance over the class of proper metric monoids is then defined along the lines Gromov's distance. The role of the $\frac{\sqrt{2}}{2}$ bound will become apparent when we prove the triangle inequality for $\Upsilon$.

\begin{definition}\label{group-GH-def}
The \emph{Gromov-Hausdorff monoid distance} $\Upsilon((G_1,\delta_1),(G_2,\delta_2))$ between two proper metric monoids $(G_1,\delta_1)$ and $(G_2,\delta_2)$ is given by:
  \begin{multline*}
    \Upsilon((G_1,\delta_1),(G_2,\delta_2)) = \\
    \min\left\{ \frac{\sqrt{2}}{2}, \inf\left\{ \varepsilon > 0 \; : \; \UIso{\varepsilon}{(G_1,\delta_1)}{(G_2,\delta_2)}{\frac{1}{\varepsilon}} \not= \emptyset \right\}\right\} \text{.}
  \end{multline*}
\end{definition}

\begin{remark}\label{almost-isoiso-exists-rmk}
  Let $(G,\delta_G)$ and $(H,\delta_H)$ be two proper monoids with respective identity elements $e_G \in G$ and $e_H \in H$. Set $\varsigma(g) = e_H$ and $\varkappa(h) = e_G$ for all $g\in G$, $h\in H$. Note that if $g,g' \in G\left[\frac{\sqrt{3}}{3}\right]$ and $h \in H\left[\frac{\sqrt{3}}{3}\right]$:
  \begin{equation*}
    \left|\delta_H(\varsigma(g)\varsigma(g'),h) - \delta_G(g g', \varkappa(h))\right| = \left|\delta_H(e_H,h) - \delta_G(g g', e_G)\right| \leq \sqrt{3} = \frac{1}{\frac{\sqrt{3}}{3}}
  \end{equation*}
and similarly with $G$ and $H$ switched so $(\varsigma,\varkappa) \in \UIso{\sqrt{3}}{(G,\delta_G)}{(H,\delta_H)}{\frac{\sqrt{3}}{3}}$. In particular $\inf\left\{ \varepsilon > 0:\UIso{\varepsilon}{(G,\delta_G)}{(H,\delta_H)}{\frac{1}{\varepsilon}}\right\} \leq \sqrt{3}$. Of course, $\Upsilon$ is always finite (no more than $\frac{\sqrt{2}}{2}$) but it is worth noticing that it is built from finite quantities only.
\end{remark}

We now proceed to prove that $\Upsilon$ is indeed a metric up to isometric isomorphism of proper monoids. As a first step, we explicit properties of almost isometric isometries, which make apparent their relationship with the desired algebraic and isometric properties. We begin with a simple observation. If $(G,\delta)$ is a metric monoid with identity element $e$, and if $g,g' \in G[r]$ for some $r\geq 0$, then $\delta(gg',e)\leq\delta(gg',g) + \delta(g,e) \leq 2r$ by left invariance of $\delta$. Therefore, the following definition makes sense.

\begin{definition}\label{near-iso-def}
Let $(G_1,\delta_1)$ and $(G_2,\delta_2)$ be two metric monoids and $r\geq 0$, $\varepsilon \geq 0$. An ordered pair $\varsigma_1:G_1[r+\varepsilon]\rightarrow G_2$ and $\varsigma_2 : G_2[r+\varepsilon]\rightarrow G_1$ of functions is called \emph{$\varepsilon$-near $r$-local isometric isomorphism} when, for all $\{j,k\} = \{1,2\}$:
\begin{enumerate}
\item $\forall g,g' \in G_j\left[\frac{r}{2}\right] \quad \delta_k(\varsigma_j(g)\varsigma_j(g'),\varsigma_j(g g')) \leq \varepsilon$,
\item $\forall g,g' \in G_j[r] \quad \left|\delta_k(\varsigma_j(g),\varsigma_j(g')) - \delta_j(g,g')\right| \leq \varepsilon$,
\item $\forall g \in G_j[r] \quad \delta_j(\varsigma_k(\varsigma_j(g)),g) \leq \varepsilon$,
\item $\delta_k(\varsigma_j(e_j),e_k) \leq \varepsilon$.
\end{enumerate}
If $\varsigma_1$ and $\varsigma_2$ map identity elements to identity elements, then the near isometric isomorphism is called \emph{unital}.
\end{definition}

We now relate our two notions of near isometries and almost isometries. Informally, Definition (\ref{almost-iso-def}) and Definition (\ref{near-iso-def}) provide two different yet tightly related means to measure the same general property --- how far a pair of maps is from being an isometric isomorphism and its inverse. Our motivation for having both notions at our disposal is that Definition (\ref{almost-iso-def}) is well-behaved under composition, which in turn will prove helpful for proving the triangle inequality for $\Upsilon$, while Definition (\ref{near-iso-def}) will prove helpful when we want to use various natural approximations of the notions of isometries or isomorphisms.

\begin{lemma}\label{almost-isoiso-lemma}
Let $(G_1,\delta_1)$, $(G_2,\delta_2)$ be two metric monoids and $\varepsilon \geq 0$, $r > 0$. 
\begin{enumerate}
 \item If $(\varsigma_1,\varsigma_2) \in \UIso{\varepsilon}{(G_1,\delta_1)}{(G_2,\delta_2)}{r}$ then for all $\{j,k\} = \{1,2\}$, if $r'=\max\{0,r-\varepsilon\}$ then:
\begin{enumerate}
\item $\forall g \in G_j[r] \quad \forall h \in G_k[r] \quad \left| \delta_k(\varsigma_j(g),h) - \delta_j(g,\varsigma_k(h)) \right| \leq \varepsilon$,
\item $\forall t \in [0,r] \quad \forall g \in G_j[t] \quad \varsigma_j(g) \in G_j[t + \varepsilon]$,
\item $\forall g \in G_j[r'] \quad \delta_j(\varsigma_k\circ\varsigma_j(g),g) \leq \varepsilon$,
\item $\forall g,g' \in G_j\left[\frac{r'}{2}\right] \quad \delta_k(\varsigma_j(g)\varsigma_j(g'),\varsigma_j(gg')) \leq 2\varepsilon$,
\item $\forall g,g' \in G_j[r'] \quad \left|\delta_k(\varsigma_j(g),\varsigma_j(g')) - \delta_j(g,g')\right| \leq 2\varepsilon$;
\end{enumerate}
in particular $(\varsigma_1,\varsigma_2)$ is a unital $2\varepsilon$-near $r'$-local isometric isomorphism;
\item if $(\varsigma_1,\varsigma_2)$ is an $r$-local unital $\varepsilon$-near isometric isomorphism, then:
\begin{equation*}
(\varsigma_1,\varsigma_2) \in \UIso{3\varepsilon}{(G_1,\delta_1)}{(G_2,\delta_2)}{\frac{r}{2}} \text{;}
\end{equation*}
\item if $(\varsigma_1,\varsigma_2)$ is a $\varepsilon$-near isometric isomorphism, then:
\begin{equation*}
\UIso{6\varepsilon}{(G_1,\delta_1)}{(G_2,\delta_2)}{\frac{r}{2}}\not=\emptyset\text{.}
\end{equation*}
\end{enumerate}
\end{lemma}

\begin{proof}
First, let $(\varsigma_1,\varsigma_2) \in \UIso{\varepsilon}{(G_1,\delta_1)}{(G_2,\delta_2)}{r}$.

Let $\{j,k\} = \{1,2\}$. If $g\in G_j[r]$ and $h\in G_k[r]$ then:
\begin{align*}
  \left| \delta_k(\varsigma_j(g),h) - \delta_j(g,\varsigma_k(h)) \right|  &= \left|\delta_k(\varsigma_j(g)e_k,h) - \delta_j(g e_j,\varsigma_k(h))\right| \\
   &= \left|\delta_k(\varsigma_j(g)\varsigma_j(e_j),h) - \delta_j(g e_j,\varsigma_k(h))\right| \\
    &\leq \varepsilon \text{.}
\end{align*}
Therefore, for all $g\in G_j[r]$ we have:
\begin{equation*}
  \left|\delta_k(\varsigma_j(g),e_2) - \delta_j(g,\varsigma_k(e_j))\right| = \left|\delta_j(\varsigma_j(g),e_2) - \delta_j(g,e_1)\right| \leq \varepsilon\text{,}
\end{equation*}
and therefore if $g\in G_j[t]$ then $\varsigma_j(g) \in G_j[t + \varepsilon]$ for all $t\in [0,r]$.

The rest of Assertion (1) trivially holds if $r' = 0$, so we assume $r' > 0$, i.e. $r > \varepsilon$.

We then note that if $g \in G_k[r-\varepsilon]$ then $\varsigma_j(g) \in G_j[r]$, and then:
\begin{equation}\label{near-inverse-eq}
\begin{split}
\delta_k(\varsigma_j(\varsigma_k(g)),g) &= \left| \delta_k(\varsigma_j(\varsigma_k(g)),g) - \delta_j(\varsigma_k(g),\varsigma_k(g))\right| + \delta_j(\varsigma_k(g),\varsigma_k(g)) \\
  &\leq \varepsilon + 0 = \varepsilon \text{.}
\end{split}
\end{equation}

Thus (1a), (1b) and (1c) are proven for all $\{j,k\}=\{1,2\}$.

Using Inequality (\ref{near-inverse-eq}), if $g,g'\in G_j[\frac{r - \varepsilon}{2}]$, then:
\begin{multline*}
\delta_k(\varsigma_j(g)\varsigma_j(g'),\varsigma_j(gg')) \\
\begin{split} 
&\leq \left|\delta_k(\varsigma_j(g)\varsigma_j(g'),\varsigma_j(gg')) - \delta_j(gg',\varsigma_k(\varsigma_j(gg')))\right| + \delta_j(gg',\varsigma_k(\varsigma_j(gg'))) \\
&\leq \varepsilon + \varepsilon = 2\varepsilon \text{.}
\end{split}
\end{multline*}

We can now conclude that for all $g,g' \in G_j[r-\varepsilon]$:
\begin{multline*}
\left|\delta_k(\varsigma_j(g),\varsigma_j(g')) - \delta_j(g,g')\right| \\
\begin{split}
  &\leq \left|\delta_k(\varsigma_j(g),\varsigma_j(g')) - \delta_j(g,\varsigma_k\circ\varsigma_j(g')) \right| + \left|\delta_j(g,\varsigma_k\circ\varsigma_j(g')) - \delta_j(g,g')\right|\\
  &\leq \varepsilon + \left|\delta_j(g, \varsigma_k\circ\varsigma_j(g')) - \delta_j(g,g') \right| \\
  &\leq \varepsilon + \delta_j(\varsigma_k\circ\varsigma_j(g'),g')\\
  &\leq 2\varepsilon \text{.}
\end{split}
\end{multline*}

This concludes our proof of (1).

Let now $(\varsigma_1,\varsigma_2)$ be an $r$-local $\varepsilon$-near isometric isomorphism. We note that for all $\{j,k\} = \{1,2\}$, $g,g' \in G_j[r]$ and $h\in G_k[r]$:
\begin{multline*}
  |\delta_k(\varsigma_j(g),h) - \delta_j(g,\varsigma_k(h))| \\
  \begin{split}
    &\leq |\delta_k(\varsigma_j(g),h)-\delta_k(\varsigma_j(g),\varsigma_j\circ\varsigma_k(h))| + |\delta_k(\varsigma_j(g),\varsigma_j(\varsigma_k(h))) - \delta_j(g,\varsigma_k(h))|\\
    &\leq \delta_k(h,\varsigma_j\circ\varsigma_k(h)) + \varepsilon \\
    &\leq 2 \varepsilon \text{,}
  \end{split}
\end{multline*}
so, if $g,g' \in G_j\left[\frac{r}{2}\right]$ and $h\in G_k[r]$:
\begin{multline*}
|\delta_k(\varsigma_j(g)\varsigma_j(g'),h) - \delta_j(g g', \varsigma_k(h))| \\
\begin{split}
&\leq |\delta_k(\varsigma_j(g)\varsigma_j(g'),h) - \delta_k(\varsigma_j(g g'), h)| + |\delta_k(\varsigma_j(g g'), h) - \delta_j(g g', \varsigma_k(h))| \\
&\leq \delta_k(\varsigma_j(g)\varsigma_j(g'),\varsigma_j(g g')) + |\delta_k(\varsigma_j(g g'), h) - \delta_j(g g', \varsigma_k(h))| \\
&\leq \varepsilon + 2\varepsilon = 3\varepsilon \text{.}
\end{split}
\end{multline*}

Thus if $(\varsigma_1,\varsigma_2)$ is unital, then we have shown that:
\begin{equation*}
(\varsigma_1,\varsigma_2) \in \UIso{3\varepsilon}{(G_1,\delta_1)}{(G_2,\delta_2)}{\frac{r}{2}}\text{.}
\end{equation*}

Assume now that $(\varsigma_1,\varsigma_2)$, still an $r$-local $\varepsilon$-near isometry, is not longer unital. Define for both $j\in\{1,2\}$:
\begin{equation*}
\varsigma_j' : g \in G_j[r+\varepsilon] \mapsto \begin{cases}
e_k \text{ if $g = e_j$,}\\
\varsigma_j(g) \text{ otherwise.}
\end{cases}
\end{equation*}

Let now $\{j,k\} = \{1,2\}$. Note that of course, $\delta_j(\varsigma'_k\circ\varsigma'_j(e_j),e_j) = 0$.  Let now $g \in G_j$. If $g\not=e_j$ we check:
\begin{align*}
  \delta_k(\varsigma'_j(g)\varsigma'_j(e), \varsigma'_j(g e)) &= \delta_k(\varsigma_j(g),\varsigma_j(g)) = 0 \text{,} 
\end{align*}
and similarly, all needed properties for $(\varsigma_1',\varsigma'_2)$ to be a near isometric isomorphism are trivially met, with the following computation the only relevant one here for $g \in G_j[r]$:
\begin{multline*}
\left|\delta_k(\varsigma'_j(g),\varsigma'_j(e_j)) - \delta_j(g,e_j)\right| \\
\begin{split}
&\leq \left|\delta_k(\varsigma'_j(g),e_j) - \delta_j(\varsigma_j(g),\varsigma_j(e_j))\right| + \left|\delta_k(\varsigma_j(g),\varsigma_j(e_j)) - \delta_j(g,e_j)\right| \\
&\leq \delta_k(e_j,\varsigma_j(e_j)) + \left|\delta_k(\varsigma_j(g),\varsigma_j(e_j)) - \delta_j(g,e_j)\right| \\
&\leq 2 \varepsilon \text{.}
\end{split}
\end{multline*}

Thus, $(\varsigma'_1,\varsigma_2')$ is an $r$-local unital $2\varepsilon$-near isometric isomorphism, and therefore $(\varsigma'_1,\varsigma_2') \in \UIso{6\varepsilon}{(G_1,\delta_1)}{(G_2,\delta_2)}{\frac{r}{2}}$ as desired.
\end{proof}

Almost isometric isomorphisms behave well under composition, which is the matter of the following lemma.

\begin{lemma}\label{uiso-composition-lemma}
Let $(G_1,\delta_1)$, $(G_2,\delta_2)$ and $(G_3,\delta_3)$ be three metric monoids with respective identity elements $e_1$, $e_2$ and $e_3$. 

If:
\begin{align*}
  (\varsigma_1,\varkappa_1)\in\UIso{\varepsilon_1}{(G_1,\delta_1)}{(G_2,\delta_2)}{\frac{1}{\varepsilon_1}}
\intertext{ and } 
(\varsigma_2,\varkappa_2)\in\UIso{\varepsilon_2}{(G_2,\delta_2)}{(G_3,\delta_3)}{\frac{1}{\varepsilon_2}}
\end{align*}
for some $\varepsilon_1, \varepsilon_2 \in \left( 0, \frac{\sqrt{2}}{2} \right]$, then:
\begin{equation*}
\left(\varsigma_2\circ\varsigma_1,\varkappa_1\circ\varkappa_2\right) \in \UIso{\varepsilon_1 + \varepsilon_2}{(G_1,\delta_1)}{(G_3,\delta_3)}{\frac{1}{\varepsilon_1 + \varepsilon_2}} \text{.}
\end{equation*}
\end{lemma}

\begin{proof}
  Let $g \in G_1\left[\frac{1}{\varepsilon_1  + \varepsilon_2}\right]$. By Assertion (1b) of Lemma (\ref{almost-isoiso-lemma}), we have:
  \begin{equation*}
    \delta_2(e_2,\varsigma_1(g)) \leq \varepsilon_1 + \delta_1(\varsigma_2(e_2),g) = \varepsilon_1 + \delta_1(e_1,g) \leq \varepsilon_1 + \frac{1}{\varepsilon_1 + \varepsilon_2} \text{.}
  \end{equation*}

  We then compute:
  \begin{align*}
    \varepsilon_1 + \frac{1}{\varepsilon_1 + \varepsilon_2} - \frac{1}{\varepsilon_2} &= \frac{\varepsilon_1(\varepsilon_1 + \varepsilon_2)\varepsilon_2 + \varepsilon_2 - (\varepsilon_1 + \varepsilon_2)}{\varepsilon_2(\varepsilon_1 + \varepsilon_2)}\\
&= \frac{\varepsilon_1}{\varepsilon_2(\varepsilon_1 + \varepsilon_2)} \left((\varepsilon_1+\varepsilon_2)\varepsilon_2 - 1 \right)\text{.}
  \end{align*}
  Of course, $\frac{\varepsilon_1}{\varepsilon_2(\varepsilon_1 + \varepsilon_2)} \geq 0$. Since $\max\left\{ \varepsilon_1, \varepsilon_2 \right\} \leq \frac{\sqrt{2}}{2}$, we have $(\varepsilon_1 + \varepsilon_2)\varepsilon_2 - 1 \leq 0$. We thus conclude:
  \begin{equation*}
    \varepsilon_1 + \frac{1}{\varepsilon_1 + \varepsilon_2} \leq \frac{1}{\varepsilon_2} \text{.}
  \end{equation*}

  Consequently, $\varsigma_1(g) \in G_2\left[\frac{1}{\varepsilon_2}\right]$. Thus $\varsigma = \varsigma_2\circ\varsigma_1$ is a well-defined map from $G_1\left[\frac{1}{\varepsilon_1 + \varepsilon_2} \right]$ to $G_3$. By symmetry, $\varkappa = \varkappa_1\circ\varkappa_2$ is well-defined as well, from $G_3\left[\frac{1}{\varepsilon_1 + \varepsilon_2}\right]$ to $G_1$. 

  Let $g,g' \in G_1\left[\frac{1}{\varepsilon_1+\varepsilon_2}\right]$ and $h \in G_3\left[\frac{1}{\varepsilon_1+\varepsilon_2}\right]$. We then compute:
  \begin{multline*}
    \left|\delta_3(\varsigma(g)\varsigma(g'),h) - \delta_1(gg',\varkappa(h))\right| \\
    \begin{split}
      &= \left|\delta_3(\varsigma_2(\varsigma_1(g))\varsigma_2(\varsigma_1(g')),h) - \delta_1(gg',\varkappa(h))\right|\\
      &\leq \left|\delta_3(\varsigma_2(\varsigma_1(g))\varsigma_2(\varsigma_1(g')),h) - \delta_2(\varsigma_1(g)\varsigma_1(g'),\varkappa_2(h))\right|\\
      &\quad+ \left|\delta_2(\varsigma_1(g)\varsigma_1(g'),\varkappa_2(h)) - \delta_1(gg',\varkappa_1(\varkappa_2(h)))\right|\\
      &\leq \varepsilon_1 + \varepsilon_2 \text{.}
    \end{split}
\end{multline*}

An analogue computation would also show that:
\begin{multline*}
  \forall g,g' \in G_3\left[\frac{1}{\varepsilon_1 + \varepsilon_2}\right] \; \forall h \in G_1\left[\frac{1}{\varepsilon_1 + \varepsilon_2}\right] \\
  \left|\delta_1(\varkappa(g)\varkappa(g'),h) - \delta_3(g g', \varsigma(h))\right| \leq \varepsilon_1 + \varepsilon_2 \text{.}
\end{multline*}

Moreover $\varsigma(e_1) = e_3$ and $\varkappa(e_3) = e_1$. Thus:
\begin{equation*}
(\varsigma,\varkappa)\in\UIso{\varepsilon_1 + \varepsilon_2}{(G_1,\delta_1)}{(G_3,\delta_3)}{\frac{1}{\varepsilon_1 + \varepsilon_2}}
\end{equation*}
 as desired.
\end{proof}

We now prove that $\Upsilon$ as defined in Definition (\ref{group-GH-def}) is indeed a metric up to isometric isomorphism on the class of proper monoids. We note that the reason we work with proper, rather than more general metric, monoids, is precisely to obtain the following coincidence property for $\Upsilon$.

\begin{theorem}\label{upsilon-metric-thm}
For any proper metric monoids $(G_1,\delta_1)$, $(G_2,\delta_2)$ and $(G_3,\delta_3)$:
\begin{enumerate}
\item $\Upsilon((G_1,\delta_1),(G_2,\delta_2)) \leq \frac{\sqrt{2}}{2}$,
\item $\Upsilon((G_1,\delta_1),(G_2,\delta_2)) = \Upsilon((G_2,\delta_2),(G_1,\delta_1))$,
\item $\Upsilon((G_1,\delta_1),(G_3,\delta_3)) \leq \Upsilon((G_1,\delta_1),(G_2,\delta_2)) + \Upsilon((G_2,\delta_2),(G_3,\delta_3))$,
\item If $\Upsilon((G_1,\delta_1),(G_2,\delta_2)) = 0$ if and only if there exists a monoid isometric isomorphism from $(G_1,\delta_1)$ to $(G_2,\delta_2)$.
\end{enumerate}
In particular, $\Upsilon$ is a metric up to metric group isometric isomorphism on the class of proper metric groups.
\end{theorem}

\begin{proof}
  The symmetry of $\Upsilon$ follows obviously from the symmetry of the definition of almost isometric isometries. Moreover our metric is bounded above by $\frac{\sqrt{2}}{2}$ by definition as well.

  We now prove the triangle inequality. Let $\upsilon_1 = \Upsilon((G_1,\delta_1),(G_2,\delta_2))$ and $\upsilon_2 = \Upsilon((G_2,\delta_2),(G_3,\delta_3))$. We first observe that if $\max\{\upsilon_1,\upsilon_2\} = \frac{\sqrt{2}}{2}$ then the triangle inequality immediately holds. So we now assume $\max\{ \upsilon_1, \upsilon_2 \} < \frac{\sqrt{2}}{2}$.

Let $\varepsilon > 0$ so that $\max\{\upsilon_1+\frac{\varepsilon}{2},\upsilon_2+\frac{\varepsilon}{2}\}<\frac{\sqrt{2}}{2}$. 

Let $(\varsigma_1,\varkappa_1) \in \UIso{\upsilon_1 + \frac{\varepsilon}{2}}{(G_1,\delta_1)}{(G_2,\delta_2)}{\frac{1}{\upsilon_1 + \frac{\varepsilon}{2}}}$. Similarly, let $(\varsigma_2,\varkappa_2)\in \UIso{\upsilon_2 + \frac{\varepsilon}{2}}{(G_2,\delta_2)}{(G_3,\delta_3)}{\frac{1}{\upsilon_2 + \frac{\varepsilon}{2}}}$. By Lemma (\ref{uiso-composition-lemma}), if $\varsigma = \varsigma_2\circ\varsigma_1$ and $\varkappa = \varkappa_1\circ\varkappa_2$ then $(\varsigma,\varkappa) \in \UIso{\upsilon_1 + \upsilon_2 + \varepsilon}{(G_1,\delta_1)}{(G_3,\delta_3)}{\frac{1}{\upsilon_1 + \upsilon_2 + \varepsilon}}$. Therefore by Definition (\ref{group-GH-def}), we conclude:
\begin{equation*}
\Upsilon((G_1,\delta_1),(G_3,\delta_3)) \leq \upsilon_1 + \upsilon_2 + \varepsilon \text{.}
\end{equation*} 
As $\varepsilon > 0$ is arbitrary, we conclude:
\begin{equation*}
\Upsilon((G_1,\delta_1),(G_3,\delta_3)) \leq \upsilon_1 + \upsilon_2 = \Upsilon((G_1,\delta_1),(G_2,\delta_2)) + \Upsilon((G_2,\delta_2),(G_3,\delta_3)) \text{,}
\end{equation*}
as desired.

We now prove that $\Upsilon((G,\delta_G),(H,\delta_H)) = 0$ implies the existence of an isometric monoid isomorphism between two proper monoids $(G,\delta_G)$ and $(H,\delta_H)$.

Let $S$ be a countable dense subset of $G$. Let $G_0$ be the submonoid generated in $G$ by $S$, which is also countable as it is the set of all finite products of elements in $S$. Similarly, let $H_0$ be a countable dense monoid of $H$.

Assume now that $\Upsilon((G,\delta_G),(H,\delta_H)) = 0$. For each $n\in\N$, let:
\begin{equation*}
(\varsigma_n,\varkappa_n) \in \UIso{\frac{1}{n+1}}{(G,\delta_G)}{(H,\delta_H)}{n+1}\text{.}
\end{equation*}
To ease our notation, for any $n\in\N$ and $g \in G\setminus G[n+1]$ and $h\in H \setminus H[n+1]$, we set $\varsigma_n(g) = e_2$ and $\varkappa_n(h) = e_1$.

Let $g\in G_0$ and $j : \N\rightarrow\N$ be any strictly increasing function. By Assertion (1b) of Lemma (\ref{almost-isoiso-lemma}), $\varsigma_{j(n)}(g) \in H[\delta(g,e) + 1]$ for all $n\in\N$. Since $(H,\delta_H)$ is proper, the set $H[\delta(g,e)+1]$ is compact, and thus $(\varsigma_{j(n)}(g))_{n\in\N}$ admits a convergent subsequence. Thus, as $G_0$ is countable, a diagonal argument shows that there exists a strictly increasing sequence $j$ of natural numbers such that, for all $g\in G_0$, the sequence $(\varsigma_{j(n)}(g))_{n\in\N}$ converges to a limit we denote by $\varsigma(g)$.

By Assertion (1e) of  Lemma (\ref{almost-isoiso-lemma}), the map $\varsigma$ is an isometry on $G_0$, since for all $g,g' \in G_0$:
\begin{align*}
  \left|\delta_H(\varsigma(g),\varsigma(g')) - \delta_G(g,g')\right| &\leq \limsup_{n\rightarrow\infty} \left|\delta_H(\varsigma_{n}(g),\varsigma_{n}(g')) - \delta_G(g,g')\right| \\
    &= 0\text{.}
\end{align*}

Now, as a uniformly continuous function over the dense subset $G_0$ of the complete metric space $G$,  the map $\varsigma$ admits a unique uniformly continuous extension to $G$ which we still denote by $\varsigma$. It is immediate that $\varsigma$ is an isometry. Moreover, let $g \in G$ and $\varepsilon > 0$. There exists $g' \in G_0$ with $\delta_G(g,g') < \frac{\varepsilon}{4}$. Let $N_0 \in \N$ be chosen such that $g' \in G[N_0]$. There is $N_1\in\N$ with $N_1 \geq N_0$ such that for all $n\geq N_1$, we have $\delta_H(\varsigma(g'),\varsigma_{j(n)}(g')) < \frac{\varepsilon}{4}$. Last, there exists $N_2 \in \N$ such that if $n\geq N_2$ then $\frac{2}{n+1} \leq \frac{\varepsilon}{4}$. We thus have for all $n\geq \max\{N_1,N_2\}$, by Assertion (1e) of Lemma (\ref{almost-isoiso-lemma}) applied to  $\varsigma_{j(n)}$, and since $\varsigma$ is an isometry:
\begin{align*}
  \delta_H(\varsigma(g),\varsigma_{j(n)}(g)) &\leq \delta_H(\varsigma(g),\varsigma(g')) + \delta_H(\varsigma(g'),\varsigma_{j(n)}(g')) + \delta_H(\varsigma_{j(n)}(g'),\varsigma_{j(n)}(g)) \\
&\leq \delta_G(g,g') + \frac{\varepsilon}{4} + \delta_G(g,g') +  \frac{\varepsilon}{4}\\
&\leq \varepsilon \text{.}
\end{align*}
Hence $\varsigma(g)$ is the limit of $(\varsigma_n(g))_{n\in\N}$ for all $g\in G$.

Now, by the same argument, there exists an isometric map $\varkappa : H \rightarrow G$ and a strictly increasing $k : \N \rightarrow  \N$ such that for all $h\in H$, the sequence $\varkappa_{j\circ k(n)}(h)$ converges to $\varkappa(h)$.

Let now $g,g' \in G$ and $h \in H$. By Definition and since the multiplication of $H$ is continuous on $H\times H$:
\begin{multline*}
  \left|\delta_H(\varsigma(g)\varsigma(g'),h) - \delta_G(gg', \varkappa(h))\right| \\
  \begin{split}
    &= \lim_{n\rightarrow\infty} \left|\delta_H(\varsigma_{j\circ k(n)}(g)\varsigma_{j\circ k(n)}(g'),h) - \delta_G(gg', \varkappa_{j\circ k(n)}(h))\right|\\
    &\leq \limsup_{n\rightarrow\infty} \frac{1}{n+1} = 0 \text{.}
  \end{split}
\end{multline*}

By Assertion (1d) of Lemma (\ref{almost-isoiso-lemma}), we conclude that $\varsigma$ is a monoid morphism. We also conclude that $\varsigma$ and $\varkappa$ are inverse of each other by Assertion (1c) of Lemma (\ref{almost-isoiso-lemma}).

Thus, $\Upsilon((G,\delta_G),(H,\delta_H)) = 0$ implies that there exists an isometric monoid isomorphism $\varsigma : G \rightarrow  H$ as desired. 

Conversely, if there exists an isometric monoid isomorphism $\varsigma$ from $(G,\delta_G)$ onto $(H,\delta_H)$ then, for any $\varepsilon > 0$, we observe that $\left(\varsigma,\varsigma^{-1}\right)\in \UIso{\varepsilon}{(G,\delta_G)}{(H,\delta_H)}{\frac{1}{\varepsilon}}$ so $\Upsilon((G,\delta_G),(H,\delta_H)) \leq \varepsilon$. Therefore $\Upsilon((G,\delta_G),(H,\delta_H)) = 0$.

We conclude with the observation that a morphism of monoid between groups is in fact a group morphism. This completes our proof.
\end{proof}

Our intent is to construct a covariant Gromov-Hausdorff distance for proper monoids, so we expect in particular that $\Upsilon$ dominates the Gromov-Hausdorff distance. We now prove that it is indeed the case. We recall from \cite{Gromov81} the definition of the Gromov-Hausdorff distance between pointed proper metric spaces:

\begin{definition}[{\cite{Gromov81}}]
If $(X,d_X)$ and $(Y,d_Y)$ are two proper metric spaces, and if $x\in X$ and $y\in Y$, then the \emph{Gromov-Hausdorff distance}  $\mathrm{GH}((X,d_X,x),(Y,d_Y,y))$ between $(X,d_X,x)$ and $(Y,d_Y,y)$ is:
\begin{equation*}
\max\left\{\frac{1}{2}, \inf\left\{ \varepsilon > 0 \middle\vert \begin{array}{l}
\text{there exists a metric $d$ on $X\coprod Y$ such that:}\\
\text{$d$ restricts to $d_X$ on $X\times X$ and $d_Y$ on $Y\times Y$,}\\
X\left[x,\frac{1}{\varepsilon}\right] \subseteq_{\varepsilon}^d Y \text{ and }Y\left[y,\frac{1}{\varepsilon}\right]\subseteq_{\varepsilon}^d X
\end{array}
  \right\} \right\} \text{,}
\end{equation*}
where for any two subsets $A$ and $B$ of a metric space $(E,d)$ and any $\varepsilon\geq 0$, the notation $A\subseteq_\varepsilon^d B$ is meant for:
\begin{equation*}
  \forall a\in A \quad \exists b \in B \quad d(a,b) \leq \varepsilon \text{,}
\end{equation*}
and $X\coprod Y$ is the disjoint union of $X$ and $Y$ (i.e. the coproduct in the category of sets).
\end{definition}

We now prove that $\Upsilon$ dominates the Gromov-Hausdorff between the underlying metric spaces, with base point chosen to be the units.

\begin{theorem}
If $(G,\delta_G)$ and $(H,\delta_H)$ are two proper metric monoids with respective units $e_G$ and $e_H$ then:
\begin{equation*}
\mathrm{GH}((G,\delta_G,e_G),(H,\delta_H,e_H)) \leq \Upsilon((G,\delta_G),(H,\delta_H)) \text{.}
\end{equation*}
\end{theorem}

\begin{proof}
We work under the assumption that $\Upsilon((G,\delta_G),(H,\delta_H)) < \frac{\sqrt{2}}{2}$, as otherwise our conclusion trivially holds. Let $(\varsigma,\varkappa) \in \UIso{\varepsilon}{(G,\delta_G)}{(H,\delta_H)}{\frac{1}{\varepsilon}}$ for $\varepsilon > \Upsilon((G,\delta_G),(H,\delta_H))$. We define $d(x,y) \in \R$ for all $x,y \in G\coprod H$ (which is just a set, not endowed with any algebraic structure) by setting:
\begin{equation*}
  \begin{cases}
    \delta_G(x,y) \text{ if $x,y \in G$,}\\
    \delta_H(x,y) \text{ if $x,y \in H$,}\\
    \varepsilon + \inf\left\{\begin{array}{l}
                               \delta_G(x,g) + \delta_H(\varsigma(g),y),\\
                               \delta_G(x,\varkappa(h)) + \delta_H(h,y)
                             \end{array} 
                             \middle\vert g \in G[\frac{1}{\varepsilon}], h \in H\left[\frac{1}{\varepsilon}\right] \right\} \text{ if $x\in G, y\in H$,}\\
                           d(y,x) \text{ if $x\in H$, $y\in G$.}
   \end{cases}
\end{equation*}
We now prove that $d$ is a metric on $G\coprod H$. By construction, $d$ is symmetric. Moreover, $x\in G$ and $y\in H$ then $d(x,y) \geq\varepsilon > 0$ and $d(y,x)\geq\varepsilon > 0$, so $d(x,y) = 0$ implies $x,y \in G$ or $x,y \in H$, which in both cases then implies $x=y$. Of course, $d(x,x) = 0$ for any $x\in G\coprod H$. We are left to show the triangular inequality. Let $x,y,z\in G\coprod H$. We have several cases to consider. First, assume that $x \in G$ and $y,z \in H$. If $g \in G\left[\frac{1}{\varepsilon}\right]$ then:
\begin{align*}
d(x,z) &\leq \varepsilon + \delta_G(x,g) + \delta_H(\varsigma(g),z) \\
&\leq  \varepsilon + \delta_G(x,g) + \delta_H(\varsigma(g),y) + \delta_H(y,z) \\
&= \varepsilon + \delta_G(x,g) + \delta_H(\varsigma(g),y) + d(y,z) \text{,}
\end{align*}
and similarly if $h \in H\left[\frac{1}{\varepsilon}\right]$ then:
\begin{align*}
  d(x,z) &\leq \varepsilon + \delta_G(x,\varkappa(h)) + \delta_H(h,z) \\
&\leq  \varepsilon + \delta_G(x,\varkappa(h)) + \delta_H(h,y) + \delta_H(y,z) \\
&= \varepsilon + \delta_G(x,\varkappa(h)) + \delta_H(h,y) + d(y,z) \text{.}
\end{align*}
so taking the infimum over $g \in G\left[\frac{1}{\varepsilon}\right]$ and $h\in H\left[\frac{1}{\varepsilon}\right]$, we get:
\begin{equation*}
d(x,z) \leq d(x,y) + d(y,z) \text{.}
\end{equation*}
By symmetry, we also have dealt with the case where $x,y\in H$ and $z \in G$. Now, assume instead that $x,y \in G$ and $z\in H$. If $g \in G\left[\frac{1}{\varepsilon}\right]$ then:
\begin{align*}
d(x,z) &\leq \varepsilon + \delta_G(x,g) + \delta_H(\varsigma(g),z) \\
&\leq \varepsilon + \delta_G(x,y) + \delta_G(y,g) + \delta_H(\varsigma(g),z) \\
&\leq  d(x,y) + \varepsilon + \delta_G(y,g) + \delta_H(\varsigma(g),z) \text{.}
\end{align*}
A similar computation shows $d(x,z) \leq d(x,y) + \varepsilon + \delta_G(y,\varkappa(h)) + \delta_H(h,z)$ for all $h\in H\left[\frac{1}{\varepsilon}\right]$. Hence once again, we conclude by definition of $d$ that:
\begin{equation*}
d(x,z) \leq d(x,y) + d(y,z) \text{.}
\end{equation*}
Again by symmetry, we now have dealt with $x\in H$ and $y,z \in G$. Last, assume that $x,z \in G$ and $y \in H$. Let $g,h\in G\left[\frac{1}{\varepsilon}\right]$. By Assertion (1e) of Lemma (\ref{almost-isoiso-lemma}), we then compute:
\begin{align*}
d(x,z) &= \delta_G(x,z) \\
&\leq \delta_G(x,g) + \delta_G(g,h) + \delta_G(h,z) \\
&\leq \delta_G(x,g) + \delta_H(\varsigma(g),\varsigma(h)) + \delta_G(h,z) + 2\varepsilon \\
&\leq \delta_G(x,g) + \delta_H(\varsigma(g),y) + \delta_H(y,\varsigma(h)) + \delta_G(h,z) + 2\varepsilon\\
&\leq \left(\varepsilon + \delta_G(x,g) + \delta_H(\varsigma(g),y)\right) + \left(\varepsilon + \delta_H(y,\varsigma(h)) + \delta_G(h,z)\right) \text{.}
\end{align*}
Similarly, if $g,h\in H\left[\frac{1}{\varepsilon}\right]$ then:
\begin{equation*}
  d(x,z) \leq \left(\varepsilon + \delta_G(x,\varkappa(g)) + \delta_H(g,y)\right) + \left(\varepsilon + \delta_H(y,h) + \delta_G(\varkappa(h),z)\right) \text{.}
\end{equation*}
We thus conclude, by taking the infimum, that $d(x,z)\leq d(x,y) + d(y,z)$. The same reasoning applies if $x,z\in H$ and $y \in G$. Last, if $x,y,z\in G$ or $x,y,z \in H$ then $d(x,z)\leq d(x,y) + d(y,z)$ by construction of $d$ since $\delta_G$ and $\delta_H$ are indeed metrics. Hence we have proven that $d$ is a metric on $G\coprod H$. Moreover, it is immediate by construction that $d$ restricts to $\delta_G$ on $G$ and $\delta_H$ on $H$.

Let now $x \in G[\frac{1}{\varepsilon}]$. Then:
\begin{equation*}
d(x,\varsigma(x)) \leq \varepsilon + \delta_G(x,x) + \delta_H(\varsigma(x),\varsigma(x)) \leq \varepsilon \text{.}
\end{equation*}
Similarly:
\begin{equation*}
d(\varkappa(x),x) \leq \varepsilon + \delta_G(\varkappa(x),\varkappa(x)) + \delta_H(x,x) \leq \varepsilon \text{.}
\end{equation*}
Therefore $G\left[\frac{1}{\varepsilon}\right]$ lies within $\varepsilon$ of $H$, and conversely $H\left[\frac{1}{\varepsilon}\right]$ lies within $\varepsilon$ of $G$. So by definition:
\begin{equation*}
  \mathrm{GH}((G,\delta_G,e_G),(H,\delta_H,e_H)) \leq \varepsilon\text{,}
\end{equation*}
from which our theorem follows as $\varepsilon > \Upsilon((G,\delta_G),(H,\delta_H))$ is arbitrary.
\end{proof}

\begin{remark}
We note in passing that if $(G,\delta_G)$ and $(H,\delta_H)$ are \emph{sets} with base points $e_G$ and $e_H$ --- and if by abuse of notation, we omit the base point from our notations for closed balls centered at these points --- and if, for some $\varepsilon>0$, we have $\mathrm{GH}((G,\delta_G,e_G),(H,\delta_H,e_H)) < \varepsilon$ ,then there exists a metric $d$ on $G\coprod H$, whose restriction to $G$ is $\delta_G$, whose restriction to $H$ is $\delta_H$, and such that $G\left[\frac{1}{\varepsilon}\right] \subseteq^d_\varepsilon H$ and $H\left[\frac{1}{\varepsilon}\right]\subseteq^d_\varepsilon G$.

If for all $x \in G\left[\frac{1}{\varepsilon}\right]$  we choose $\varsigma(x)$ such that with $d(x,\varsigma(x)) \leq \varepsilon$, and for $h\in H\left[\frac{1}{\varepsilon}\right]$, we choose $\varkappa(x)$ such that $d(x,\varkappa(x)) \leq \varepsilon$, then for all $g\in G\left[\frac{1}{\varepsilon}\right]$:
\begin{align*}
\left| \delta_H(\varsigma(g),\varsigma(g')) - \delta_G(g,g') \right| &\leq |d(\varsigma(g),\varsigma(g')) - d(\varsigma(g),g')| + |d(\varsigma(g),g') - d(g,g')|\\
&\leq d(\varsigma(g'),g') + d(\varsigma(g),g) \leq 2 \varepsilon \text{,}
\end{align*}
and similarly for all $h\in H\left[\frac{1}{\varepsilon}\right]$ and $\varkappa$. We can also check that:
\begin{align*}
\delta_G(g,\varkappa\circ\varsigma(g)) &= d(g,\varkappa(\varsigma(g)))  \\
&\leq d(g,\varsigma(g)) + d(\varsigma(g),\varkappa(\varsigma(g))) \leq 2\varepsilon \text{.}
\end{align*}
Hence we recover Conditions (1) and (3) of Definition (\ref{near-iso-def}) --- we have proven in passing that we can use these two conditions to define classes of $2$-tuple of functions which allow to define the topology of the Gromov-Hausdorff for pointed, proper metric spaces.
\end{remark}

\section{The Covariant Propinquity}

The covariant propinquity is defined on a class of dynamical systems, namely, for our purpose, on {\qcms s} endowed with a strongly continuous action of a proper monoid by Lipschitz positive unital linear endomorphisms. As our work is organized around metric notions, we immediately include the metric on monoids as part of our definition of a Lipschitz dynamical system, even if the definition itself only requires a topological monoid.

\begin{notation}
  Let $(\A,\Lip_\A)$ and $(\B,\Lip_\B)$ be two {\qcms s}. If $\pi : \A \rightarrow \B$ is a unital positive linear map, then:
  \begin{equation*}
    \dil{\pi} = \inf\left\{ k > 0 : \forall a \in \sa{\A} \quad \Lip\circ\pi(a) \leq k \Lip(a) \right\} \text{.}
  \end{equation*}
  By definition, $\dil{\pi} < \infty$ if and only if $\pi$ is a Lipschitz linear map.
\end{notation}

\begin{definition}
  Let $F$ be a permissible function. A \emph{Lipschitz dynamical $F$-system} $(\A,\Lip,G,\delta,\alpha)$ is a {\qcms{F}} $(\A,\Lip)$ and a proper monoid $(G,\delta)$, together with an action $\alpha$ by positive unital maps (i.e. a morphism from $G$ to the monoid of positive linear maps) such that:
\begin{enumerate}
  \item $\alpha$ is strongly continuous: for all $a\in\A$ and $g \in G$, we have:
    \begin{equation*}
      \lim_{h\rightarrow g} \norm{\alpha^h(a) - \alpha^g(a)}{\A} = 0\text{,}
    \end{equation*}
  \item $g\in G\mapsto \dil{\alpha^g}$ is locally bounded: for all $\varepsilon>0$ and $g\in G$ there exist $D > 0$ and a neighborhood $U$ of $g$ in $G$ such that if $h\in U$ then $\dil{\alpha^h} \leq D$.
\end{enumerate}

A \emph{Lipschitz $C^\ast$-dynamical $F$-system} $(\A,\Lip,G,\delta,\alpha)$ is a Lipschitz dynamical system where $G$ is a proper group and $\alpha^g$ is a Lipschitz unital *-automorphism for all $g \in G$.
\end{definition}

The class of Lipschitz dynamical systems include various sub-classes of interest, from group actions by full quantum isometries, to actions by completely positive maps, to actions by Lipschitz automorphisms or even unital endomorphisms.

There is a natural choice of morphisms between two Lipschitz dynamical systems $\mathds{A} = (\A,\Lip_\A,G,\delta_G,\alpha)$ and $\mathds{B} = (\B,\Lip_\B,H,\delta_H,\beta)$. A pair $(\pi,\varsigma)$ is a \emph{morphism} from $\mathds{A}$ to $\mathds{B}$ when $\pi : \A\rightarrow\B$ is a *-morphism such that $\pi(\dom{\Lip_\A})\subseteq \dom{\Lip_\B}$, $\varsigma : G\rightarrow H$ is a Lipschitz map and a monoid morphism, and for all $g \in G$ we have $\pi\circ\alpha^g = \beta^{\varsigma(g)}\circ\pi$. By \cite{Latremoliere14}, $\pi$ could as well be a unital *-morphism for which there exists $k \geq 0$ such that $\Lip_\B\circ\pi \leq k \Lip_\A$, or by \cite{Rieffel00}, for which $\pi^\ast$ is a $k$-Lipschitz map from $(\StateSpace(\B),\Kantorovich{\Lip_\B})$ to $(\StateSpace(\A),\Kantorovich{\Lip_\A})$, as these three notions coincide. On the other hand, we could relax the requirement that $\varsigma$ be Lipschitz to something like uniformly continuous or even continuous, though this would not fully capture the metric structure of our monoids. Alternatively, we could strengthen our requirement on $\varsigma$ and ask for it to be an isometry, and similarly we could require $\pi$ to be a quantum isometry. All these choices lead to various nested categories whose objects are Lipschitz dynamical systems.

Now, the notion of isomorphism for Lipschitz dynamical system, for our purpose, is clear, as we wish to preserve all the involved structures. Our covariant propinquity will be a metric up to the following notion of isomorphism:
\begin{definition}
  An \emph{equivariant quantum full isometry} $(\pi,\varsigma) : \mathds{A} \rightarrow \mathds{B}$ is given by a full quantum isometry $\pi : (\A,\Lip_\A) \rightarrow (\B,\Lip_\B)$ and a monoid isometric isomorphism $\varsigma : G\rightarrow H$ such that for all $g \in G$:
\begin{equation*}
  \pi\circ\alpha^g = \beta^{\varsigma(g)}\circ\pi \text{.}
\end{equation*}
\end{definition}

The construction of the covariant propinquity begins with generalizing the notion of a tunnel between {\qcms s}, as defined in \cite{Latremoliere13b,Latremoliere14} for our construction of the Gromov-Hausdorff propinquity, to our class of Lipschitz dynamical systems. Notably, the needed changes are minimal.

\begin{definition}\label{equi-tunnel-def}
Let $\varepsilon > 0$ and $F$ be a permissible function. Let $(\A_1,\Lip_1,G_1,\delta_1,\alpha_1)$ and $(\A_2,\Lip_2,G_2,\delta_2,\alpha_2)$ be two Lipschitz dynamical $F$-systems. Let $e_1$ and $e_2$ be the identity elements of $G_1$ and $G_2$ respectively. A \emph{$\varepsilon$-covariant $F$-tunnel}:
\begin{equation*}
\tau = (\D,\Lip_\D,\pi_1,\pi_2,\varsigma_1,\varsigma_2)
\end{equation*}
from $(\A_1,\Lip_1,G_1,\delta_1,\alpha_1)$ to $(\A_2,\Lip_2,G_2,\delta_2,\alpha_2)$ is given by 
\begin{equation*}
  (\varsigma_1,\varsigma_2) \in \UIso{\varepsilon}{(G_1,\delta_1)}{(G_2,\delta_2)}{\frac{1}{\varepsilon}} \text{,}
\end{equation*}
an $F$-{\qcms} $(\D,\Lip_\D)$, and two quantum isometries $\pi_1 : (\D,\Lip_\D) \twoheadrightarrow (\A_1,\Lip_1)$ and $\pi_2 : (\D,\Lip_\D) \twoheadrightarrow (\A_2,\Lip_2)$.
\end{definition}

\begin{remark}
  If $\tau$ is an $\varepsilon$-covariant tunnel then it is also an $\eta$-covariant tunnel for any $\eta \geq \varepsilon$.
\end{remark}

\begin{remark}
  If $(\D,\Lip,\pi,\rho,\varsigma,\varkappa)$ is a covariant tunnel from $(\A,\Lip_\A,G,\delta_G,\alpha)$ to $(\B,\allowbreak \Lip_\B,H,\delta_H,\beta)$, then $(\D,\Lip,\pi,\rho)$ is a tunnel from $(\A,\Lip_\A)$ to $(\B,\Lip_\B)$ in the sense of \cite{Latremoliere13b}. We also note that covariant tunnels are not constructed using a Lipschitz dynamical systems. They only involve an almost isometric isomorphism.
\end{remark}

The covariant propinquity is defined from certain quantities associated with covariant tunnels. These quantities do not depend on the quasi-Leibniz inequality. The first of these quantities arises from our work on the propinquity, applied to the underlying tunnel of a covariant tunnel.

\begin{notation}
  Let $\pi : \A \rightarrow \B$ be a positive unital linear map between two unital C*-algebras $\A$ and $\B$. We denote the dual map $\varphi \in \StateSpace(\B) \mapsto \varphi\circ\pi \in \StateSpace(\A)$ by $\pi^\ast$.
\end{notation}

\begin{notation}
  If $(E,d)$ is a metric space, then the Hausdorff distance \cite{Hausdorff} defined on the space of the closed subsets of $(E,d)$ is denoted by $\Haus{d}$. In case $E$ is a normed vector space and $d$ is the distance associated with some norm $N$, we write $\Haus{N}$ for $\Haus{d}$.
\end{notation}

\begin{definition}[{\cite[Definition 2.11]{Latremoliere14}}]\label{extent-def}
  Let $\mathds{A}_1 = (\A_1,\Lip_1,G_1,\delta_1,\alpha_1)$ and $\mathds{A}_2 = (\A_2,\Lip_2,\allowbreak G_2,\delta_2,\alpha_2)$ be two Lipschitz dynamical systems. The \emph{extent $\tunnelextent{\tau}$} of a covariant tunnel $\tau = (\D,\Lip_\D,\pi_1,\pi_2,\varsigma_1,\varsigma_2)$ from $\mathds{A}_1$ to $\mathds{A}_2$ is given as:
\begin{equation*}
\max\left\{ \Haus{\Kantorovich{\Lip}}\left(\StateSpace(\D),\pi_j^\ast\left(\StateSpace(\A_j)\right)\right) \middle\vert j\in\{1,2\} \right\} \text{.}
\end{equation*}
\end{definition}

We now introduce the new quantity in our work with covariant tunnels. The reach of a covariant tunnel, defined below, bring together the tunnel data and the almost isometric isomorphism data using an idea which generalizes \cite[Definition 3.4]{Latremoliere13b}. 

\begin{definition}\label{reach-def}
  Let $\varepsilon > 0$. Let $\mathds{A}_1 = (\A_1,\Lip_1,G_1,\delta_1,\alpha_1)$ and $\mathds{A}_2 = (\A_2,\Lip_2,G_2,\delta_2,\allowbreak \alpha_2)$ be two Lipschitz dynamical systems. The \emph{$\varepsilon$-reach $\tunnelreach{\tau}{\varepsilon}$} of a $\varepsilon$-covariant tunnel $\tau = (\D,\Lip_\D,\pi_1,\pi_2,\varsigma_1,\varsigma_2)$ from $\mathds{A}_1$ to $\mathds{A}_2$ is given as:
\begin{equation*}
\max_{\{j,k\}=\{1,2\}}\sup_{\varphi\in\StateSpace(\A_j)} \inf_{\psi\in\StateSpace(\A_k)}\sup_{g \in G_j\left[\frac{1}{\varepsilon}\right]} \Kantorovich{\Lip_\D}(\varphi\circ\alpha_j^g\circ\pi_j, \psi\circ\alpha_k^{\varsigma_j(g)}\circ\pi_k)
\end{equation*}
\end{definition}

The magnitude of a covariant tunnel summarizes all the data computed above.

\begin{definition}\label{magnitude-def}
  Let $\varepsilon > 0$. The \emph{$\varepsilon$-magnitude} $\tunnelmagnitude{\tau}{\varepsilon}$ of a $\varepsilon$-covariant tunnel $\tau$ is the maximum of its $\varepsilon$-reach and its extent:
  \begin{equation*}
    \tunnelmagnitude{\tau}{\varepsilon} = \max\left\{ \tunnelreach{\tau}{\varepsilon}, \tunnelextent{\tau} \right\} \text{.}
  \end{equation*}
\end{definition}

\begin{remark}
  If $G_1 = G_2 = \{e\}$ and using the notations of Definition (\ref{magnitude-def}), then we note that for all $\varphi \in \StateSpace(\A)$, since $\varphi\circ\pi_\A\in\StateSpace(\D)$, there exists $\psi \in \StateSpace(\B)$ such that $\Kantorovich{\Lip_1}(\varphi\circ\pi_\A,\psi\circ\pi_\B) \leq \tunnelextent{\tau}$. It is then easy to check that $\tunnelmagnitude{\tau}{\varepsilon} = \tunnelextent{\tau}$ for all $\varepsilon > 0$, and thus we just recover the extent defined in \cite{Latremoliere14} for the propinquity.
\end{remark}

\begin{remark}
  By \cite{Latremoliere13b}, there always exists $F$-tunnels between any two $F$-{\qcms s}. By Remark (\ref{almost-isoiso-exists-rmk}), we conclude that there also always exists covariant $F$-tunnels. 
\end{remark}

We first show that composition of tunnels, up to an arbitrary small error, as defined in \cite{Latremoliere14}, extends to covariant tunnels, thanks in part to our Lemma (\ref{uiso-composition-lemma}). A key aspect of this construction is to ensure that the covariant tunnels obtained by composition have the same quasi-Leibniz property as the covariant tunnels they are built from.

\begin{theorem}\label{triangle-thm}
Let $\mathds{A} = (\A,\Lip_\A,G_1,\delta_1,\alpha)$, $\mathds{B} = (\B,\Lip_\B,G_2,\delta_2,\beta)$ and $\mathds{E} = (\alg{E},\Lip_{\alg{E}},G_3,\delta_3,\gamma)$ be three Lipschitz dynamical $F$-systems for some given permissible function $F$. Let $\varepsilon_1, \varepsilon_2 \in \left(0,\frac{\sqrt{2}}{2}\right)$.

Let $\tau_1 = (\D_1,\Lip_{\D_1},\pi_1,\rho_1,\varsigma_1,\varkappa_1)$ be a $\varepsilon_1$-covariant $F$-tunnel from $\mathds{A}$ to $\mathds{B}$. Let $\tau_2 = (\D_2,\Lip_{\D_2},\pi_2,\rho_2,\varsigma_2,\varkappa_2)$ be a $\varepsilon_2$-covariant $F$-tunnel from $\mathds{B}$ to $\mathds{E}$.

Let $\varepsilon > 0$. We define for $(d_1,d_2) \in \sa{\D} = \sa{\D_1}\oplus\sa{\D_2}$:
\begin{equation*}
\Lip(d_1,d_2) = \max\left\{ \Lip_{\D_1}(d_1), \Lip_{\D_2}(d_2), \frac{1}{\varepsilon}\norm{\rho_\B(d_1) - \pi_\B(d_2)}{\B} \right\} \text{.}
\end{equation*}

We set $\varsigma=\varsigma_2\circ\varsigma_1$ and $\varkappa = \varkappa_1\circ\varkappa_2$. We also set $\eta_1 : (d_1,d_2)\in \D \mapsto \pi_\A(d_1)$ and $\eta_2:(d_1,d_2)\in\D\mapsto \rho_{\alg{E}}(d_2)$. If:
\begin{equation*}
\tau_1\circ_\varepsilon \tau_2 = (\D_1\oplus\D_2,\Lip,\eta_1,\eta_2,\varsigma,\varkappa)
\end{equation*}
then $\tau_1\circ_\varepsilon \tau_2$ is a $(\varepsilon_1 + \varepsilon_2)$-covariant $F$-tunnel from $(\A,\Lip_\A)$ to $(\alg{E},\Lip_{\alg{E}})$ such that:
\begin{equation*}
\tunnelmagnitude{\tau_1\circ_\varepsilon\tau_2}{\varepsilon_1 + \varepsilon_2} \leq \tunnelmagnitude{\tau_1}{\varepsilon_1} + \tunnelmagnitude{\tau_2}{\varepsilon_2} + \varepsilon \text{.}
\end{equation*}
\end{theorem}

\begin{proof}
We will write $\tau = \tau_1\circ_\varepsilon\tau_2$ and $\D = \D_1\oplus\D_2$ in this proof to lighten our notation. The quadruple $(\D,\Lip,\eta_1,\eta_2)$ is indeed an $F$-tunnel from $(\A,\Lip_\A)$ to $(\alg{E},\Lip_{\alg{E}})$ with $\tunnelextent{\tau} \leq \tunnelextent{\tau_1} + \tunnelextent{\tau_2} + \varepsilon$ by \cite[Theorem 3.1]{Latremoliere14}.

Now, we let $\varsigma = \varsigma_2\circ\varsigma_1$ and let $\varkappa = \varkappa_1\circ\varkappa_2$. It then follows by Lemma (\ref{uiso-composition-lemma}) that:
\begin{equation*}
  (\varsigma,\varkappa) \in \UIso{\varepsilon_1 + \varepsilon_2}{(G_1,\delta_1)}{(G_3,\delta_3)}{\frac{1}{\varepsilon_1 + \varepsilon_2}} \text{,}
\end{equation*}
and therefore
\begin{equation*}
\left(\D_1 \oplus \D_2, \Lip, \eta_1, \eta_2, \varsigma, \varkappa\right)
\end{equation*}
is an $(\varepsilon_1 + \varepsilon_2)$-covariant $F$-tunnel. It remains to compute the reach of $\tau$ for $r = \frac{1}{\varepsilon_1 + \varepsilon_2}$.

Let $\varphi \in \StateSpace(\A)$. There exists $\psi \in \StateSpace(\B)$ such that:
\begin{equation*}
  \sup_{g \in G_1[r]} \Kantorovich{\Lip_1}(\varphi\circ\alpha^g\circ\pi_\A,\psi\circ\beta^{\varsigma_1(g)}\circ\pi_\B) \leq \tunnelreach{\tau_1}{\varepsilon_1} \text{.}
\end{equation*}

There exists $\theta\in \StateSpace(\alg{E})$ such that:
\begin{equation*}
  \sup_{g \in G_2\left[\frac{1}{\varepsilon_2}\right]} \Kantorovich{\Lip_1}(\psi\circ\beta^g\circ\rho_\B,\theta\circ\gamma^{\varsigma_2(g)}\circ\pi_{\alg{E}}) \leq \tunnelreach{\tau_2}{\varepsilon_2} \text{.}
\end{equation*}
Therefore, noting $\beta^g$ is a positive unital map for all $g \in G_2$, and thus a linear map of norm $1$, and that as in Lemma (\ref{uiso-composition-lemma}), if $g \in G_1[r]$ then $\varsigma_1(g) \in G_2[r+\varepsilon_2] = G_2\left[\frac{1}{\varepsilon_2}\right]$ and:
\begin{multline*}
  \left|\varphi\circ\alpha^g\circ\pi_\A(d_1) - \theta\circ\gamma^{\varsigma(g)}\circ\rho_{\alg{E}}(d_2)\right| \\ 
  \begin{split}
    &\leq \left|\varphi\circ\alpha^g\circ\pi_\A(d_1) - \psi\circ\beta^{\varsigma_1(g)}\circ\pi_{\B}(d_1)\right|  \\
    &\quad + \left| \psi\circ\beta^{\varsigma_1(g)}\circ\pi_{\B}(d_1) - \psi\circ\beta^{\varsigma_1(g)}\circ\rho_\B(d_2) \right| \\
    &\quad + \left|\psi\circ\beta^{\varsigma_1(g)}\circ\rho_\B(d_2) - \theta\circ\gamma^{\varsigma(g)}\circ\rho_{\alg{E}}(d_2)\right| \\
    &\leq \tunnelreach{\tau_1}{\varepsilon_1} + \norm{\pi_\B(d_1) - \rho_\B(d_2)}{\B} \\
    &\quad + \left|\psi\circ\beta^{\varsigma_1(g)}\circ\rho_\B(d_2) - \theta\circ\gamma^{\varsigma_2(\varsigma_1(g))}\circ\rho_{\alg{E}}(d_2)\right|\\
    &\leq \tunnelreach{\tau_1}{\varepsilon_1} + \varepsilon + \tunnelreach{\tau_2}{\varepsilon_2} \text{.}
  \end{split}
\end{multline*}
Therefore, as the computation is symmetric in $\mathds{A}$ and $\mathds{E}$, we conclude:
\begin{equation*}
  \tunnelreach{\tau}{\varepsilon_1 + \varepsilon_2} \leq \tunnelreach{\tau_1}{\varepsilon_1} + \tunnelreach{\tau_2}{\varepsilon_2} + \varepsilon \text{.}
\end{equation*}

Hence:
\begin{equation*}
\tunnelmagnitude{\tau}{\varepsilon_1 + \varepsilon_2} \leq \tunnelmagnitude{\tau_1}{\varepsilon_1} + \tunnelmagnitude{\tau_2}{\varepsilon_2} + \varepsilon \text{.}
\end{equation*}

This concludes our proof.
\end{proof}

\begin{remark}
  The fact that the actions do not actually enter the construction of the composition of covariant tunnels is important as it allows us to work with actions by non-multiplicative maps, without worry that this would compromise the quasi-Leibniz property of the composed tunnel.
\end{remark}

As with the dual propinquity \cite{Latremoliere13b,Latremoliere14}, we can enforce additional properties on the covariant tunnels used to define the covariant propinquity. The motivation for this flexibility is that it becomes possible to make sure that when two Lipschitz dynamical systems are close for a chosen specialization of the covariant propinquity, then the tunnels has desirable properties for the problem at hand. In order to define the covariant propinquity, we require certain properties on the choice of a class of covariant tunnels --- so that our construction in this paper indeed leads to a metric up to equivariant full quantum isometry, as seen later on.

\begin{definition}\label{appropriate-def}
  Let $F$ be a permissible function. Let $\mathcal{C}$ be a nonempty class of Lipschitz dynamical $F$-systems. A class $\mathcal{T}$ of covariant $F$-tunnels is \emph{appropriate} for $\mathcal{C}$ when:
  \begin{enumerate}
    \item for all $\mathds{A},\mathds{B} \in \mathcal{C}$, there exists a $\varepsilon$-covariant tunnel from $\mathds{A}$ to $\mathds{B}$ for some $\varepsilon > 0$,
    \item if $\tau \in \mathcal{T}$, then there exist $\mathds{A},\mathds{B} \in \mathcal{C}$ such that $\tau$ is a covariant tunnel from $\mathds{A}$ to $\mathds{B}$,
    \item if $\mathds{A} = (\A,\Lip_\A,G,\delta_G,\alpha)$, $\mathds{B} = (\B,\Lip_\B,H,\delta_H,\beta)$ are elements of $\mathcal{C}$, and if there exists an equivariant full quantum isometry $(\pi,\varsigma)$ from $\mathds{A}$ to $\mathds{B}$, then:
      \begin{equation*}
        \left( \A,\Lip_\A,\mathrm{id}_\A,\pi,\varsigma,\varsigma^{-1} \right), \left( \B,\Lip_\B,\pi^{-1},\mathrm{id}_\B,\varsigma^{-1},\varsigma \right)  \in \mathcal{T}
      \end{equation*}
      where $\mathrm{id}_\A$, $\mathrm{id}_\A$ are the identity *-automorphisms of $\A$ and $\B$, respectively,
    \item if $\tau = (\D,\Lip,\pi,\rho,\varsigma,\varkappa) \in \mathcal{T}$ then $\tau^{-1} = (\D,\Lip,\rho,\pi,\varkappa,\varsigma) \in \mathcal{T}$,
    \item if $\varepsilon > 0$ and if $\tau_1,\tau_2 \in \mathcal{T}$ are $\frac{\sqrt{2}}{2}$-tunnels, then there exists $\delta \in (0,\varepsilon]$ such that $\tau_1\circ_\delta \tau_2 \in \mathcal{T}$.
  \end{enumerate}
\end{definition}

\begin{proposition}
  If $F$ is a permissible function, then the class of all $F$-covariant tunnels is appropriate for the class of all Lipschitz dynamical $F$-systems.
\end{proposition}

\begin{proof}
  We already observed that Condition (1) of Definition (\ref{appropriate-def}) holds and Condition (2) is trivial here. Condition (3) of Definition (\ref{appropriate-def}) is met since it is straightforward to check that the quintuple listed there are indeed $F$-tunnels. Condition (4) is trivial here. Last, Condition (5) of Definition (\ref{appropriate-def}) is non-trivial, and is met precisely thanks to Theorem (\ref{triangle-thm}).
\end{proof}

We summarize a recurrent assumption to many statements in the rest of this section to improve our presentation.
\begin{hypothesis}\label{hyp}
  Let $F$ be a permissible function. Let $\mathcal{C}$ be a nonempty class of Lipschitz dynamical $F$-systems and $\mathcal{T}$ be a class of $F$-tunnels appropriate for $\mathcal{C}$.
\end{hypothesis}

We are now ready to define the main object of this paper, which is a metric on the class of Lipschitz dynamical systems.

\begin{notation}
  Assume Hypothesis (\ref{hyp}). Let $\mathds{A}$ and $\mathds{B}$ in $\mathcal{C}$. Let $\varepsilon > 0$. The class of all $\varepsilon$-covariant tunnels in $\mathcal{T}$ from $\mathds{A}$ to $\mathds{B}$ is denoted as:
  \begin{equation*}
    \tunnelset{\mathds{A}}{\mathds{B}}{\mathcal{T}}{\varepsilon} \text{.}
  \end{equation*}
We write $\tunnelset{\mathds{A}}{\mathds{B}}{F}{\varepsilon}$ for the set of \emph{all} $\varepsilon$-covariant $F$-tunnels from $\mathds{A}$ to $\mathds{B}$.
\end{notation}

\begin{definition}\label{covariant-propinquity-def}
  Assume Hypothesis (\ref{hyp}). For $\mathds{A},\mathds{B} \in \mathcal{C}$, the \emph{covariant $\mathcal{T}$-propinquity} $\covpropinquity{\mathcal{T}}(\mathds{A},\mathds{B})$ is defined as:
\begin{equation*}
  \min\left\{ \frac{\sqrt{2}}{2}, \inf\left\{ \varepsilon > 0 \middle\vert \exists \tau \in \tunnelset{\mathds{A}}{\mathds{B}}{\mathcal{T}}{\varepsilon} \quad \tunnelmagnitude{\tau}{\varepsilon} \leq \varepsilon \right\} \right\} \text{.}
\end{equation*}
\end{definition}

\begin{notation}
  We write $\covpropinquity{F}$ for the covariant propinquity on the class of all Lipschitz dynamical $F$-systems, using all possible covariant $F$-tunnels.
\end{notation}

We conclude from our result on tunnels composition:

\begin{proposition}\label{pseudo-prop}
  Assume Hypothesis (\ref{hyp}). The covariant propinquity $\covpropinquity{\mathcal{T}}$ is a pseudo-metric, bounded above by $\frac{\sqrt{2}}{2}$, on the class of Lipschitz dynamical systems, and moreover, for all Lipschitz dynamical systems $\mathds{A} = (\A,\Lip_\A,G,\delta_G,\alpha)$ and $\mathds{B} = (\B,\Lip_\B,H,\delta_H,\beta)$:
  \begin{equation*}
    \min\left\{\propinquity{\mathcal{T}'}((\A,\Lip_\A),(\B,\Lip_\B)), \frac{\sqrt{2}}{2} \right\} \leq \covpropinquity{\mathcal{T}}(\mathds{A},\mathds{B}) \text{,}
  \end{equation*}
where $\mathcal{T'} = \left\{ (\D,\Lip,\pi,\rho) : \exists \tau = (\D,\Lip,\pi,\rho,\varsigma,\varkappa) \in \mathcal{T} \right\}$ is a class of tunnels appropriate with $\{(\A,\Lip_\A) : \exists (\A,\Lip_\A,G,\delta,\alpha) \in \mathcal{C}\}$.
\end{proposition}

\begin{proof}
  Symmetry is obvious by definition. Let $\mathds{A}  = (\A,\Lip_\A,G_1,\delta_1,\alpha)$, $\mathds{B} = (\B,\Lip_\B,\allowbreak G_2,\delta_2,\beta)$, and $\mathds{E} = (\alg{E},\Lip_{\alg{E}},G_3,\delta_3,\gamma)$ be three dynamical systems. If $\covpropinquity{\mathcal{T}}(\mathds{A},\mathds{B}) \geq \frac{\sqrt{2}}{2}$ or $\covpropinquity{\mathcal{T}}(\mathds{B},\mathds{E}) \geq  \frac{\sqrt{2}}{2}$ then by definition:
\begin{equation*}
  \covpropinquity{\mathcal{T}}(\mathds{A},\mathds{E})\leq \covpropinquity{\mathcal{T}}(\mathds{A},\mathds{B})  + \covpropinquity{\mathcal{T}}(\mathds{B},\mathds{E})\text{.}
\end{equation*}
Hence, we now assume that both $\covpropinquity{\mathcal{T}}(\mathds{A},\mathds{B}) < \frac{\sqrt{2}}{2}$ and $\covpropinquity{\mathcal{T}}(\mathds{B},\mathds{E}) <  \frac{\sqrt{2}}{2}$. Write $\upsilon_1 = \covpropinquity{\mathcal{T}}(\mathds{A},\mathds{B})$ and $\upsilon_2 = \covpropinquity{\mathcal{T}}(\mathds{B},\mathds{E})$.

Let $\varepsilon > 0$ such that $\max\{\upsilon_1+\frac{\varepsilon}{3}, \upsilon_2 + \frac{\varepsilon}{3} \} < \frac{\sqrt{2}}{2}$. By definition, there exist:
  \begin{equation*}
    \tau_1 \in \tunnelset{\mathds{A}}{\mathds{B}}{\mathcal{T}}{\upsilon_1+\frac{\varepsilon}{3}}\text{ and }\tau_2 \in \tunnelset{\mathds{B}}{\mathds{E}}{\mathcal{T}}{\upsilon_2+\frac{\varepsilon}{3}}
  \end{equation*}
  such that:
  \begin{equation*}
    \tunnelmagnitude{\tau_1}{\upsilon_1 + \frac{\varepsilon}{3}} \leq  \upsilon_1 + \frac{\varepsilon}{3} \text{ and }\tunnelmagnitude{\tau_2}{\upsilon_2 + \frac{\varepsilon}{3}} \leq \upsilon_2 + \frac{\varepsilon}{3} \text{.}
  \end{equation*}
As $\mathcal{T}$ is appropriate, for some $\delta \in (0,\varepsilon]$, we have $\tau = \tau_1\circ_\delta \tau_2 \in \mathcal{T}$. By Theorem (\ref{triangle-thm}), the tunnel $\tau$ is an $(\upsilon_1+\upsilon_2+\varepsilon)$-covariant tunnel $\tau$ from $\mathds{A}$ to $\mathds{E}$ such that:
\begin{equation*}
  \tunnelmagnitude{\tau}{\upsilon_1 + \upsilon_2 + \varepsilon} \leq  \tunnelmagnitude{\tau_1}{\upsilon_1 + \frac{\varepsilon}{3}} + \tunnelmagnitude{\tau_2}{\upsilon_2 + \frac{\varepsilon}{3}} + \frac{\varepsilon}{3} \leq \upsilon_1 + \upsilon_2 + \varepsilon \text{.}
\end{equation*}

Therefore by definition:
\begin{equation*}
  \covpropinquity{\mathcal{T}}(\mathds{A},\mathds{E}) \leq \upsilon_1 + \upsilon_2 + \varepsilon = \covpropinquity{\mathcal{T}}(\mathds{A},\mathds{B}) + \covpropinquity{\mathcal{T}}(\mathds{B},\mathds{E}) + \varepsilon \text{.}
\end{equation*}
As $\varepsilon > 0$ is arbitrary, this concludes our proof of the triangle inequality.

To conclude the proof of our proposition, we observe that if $\tau = (\D,\Lip,\pi_\A,\pi_\B,\allowbreak \varsigma,\varkappa)$ is a $\varepsilon$-covariant tunnel from $\mathds{A}$ to $\mathds{B}$, then in particular $\gamma = (\D,\Lip,\pi_\A,\pi_\B)$ is a tunnel in $\mathcal{T}'$ and:
\begin{equation*}
  \tunnelextent{\gamma} = \tunnelextent{\tau} \leq \tunnelmagnitude{\tau}{\varepsilon}\text{.}
\end{equation*}
It is also easy to check that $\mathcal{T}'$ is appropriate for the class of underlying {\qcms s} of $\mathcal{C}$. It then follows by definition that:
\begin{equation*}
  \propinquity{\mathcal{T}'}((\A,\Lip_\A),(\B,\Lip_\B)) \leq \inf\left\{\varepsilon > 0 \middle\vert \exists \tau \in \tunnelset{\mathds{A}}{\mathds{B}}{\mathcal{T}}{\varepsilon} \quad \tunnelmagnitude{\tau}{\varepsilon} \leq \varepsilon \right\} \text{.}
\end{equation*}
Therefore, by definition:
\begin{equation*}
\min\left\{\propinquity{\mathcal{T}'}((\A,\Lip_\A),(\B,\Lip_\B)),\frac{\sqrt{2}}{2}\right\} \leq \covpropinquity{\mathcal{T}}(\mathds{A},\mathds{B})
\end{equation*}
as desired.
\end{proof}

We now prove that distance zero for the covariant propinquity is equivalent to the existence of an equivariant full quantum isometry. We will use our work in \cite{Latremoliere13b}; in particular we recall the notion of a target set for a tunnel and a couple of their properties which we will use here.

Let $(\A,\Lip_\A)$ and $(\B,\Lip_\B)$ be two {\qcms s}. Let $\tau = (\D,\Lip_\D,\pi_\A,\pi_\B)$ be a tunnel from $(\A,\Lip_\A)$ to $(\B,\Lip_\B)$. For any $a\in \dom{\Lip_\A}$ and $l\geq\Lip_\A(a)$, the \emph{$l$-target set of $a$} is defined by:
\begin{equation*}
  \targetsettunnel{\tau}{a}{l} = \left\{ \pi_\B(d) \middle\vert d\in \sa{\D}, \Lip_\D(d) \leq l, \pi_\A(d) = a \right\} \text{.} 
\end{equation*}
Now, if $\tau = (\D,\Lip_\D,\pi_\A,\pi_\B,\varsigma,\varkappa)$ is a covariant tunnel, then for all $a\in\sa{\A}$ and $l \geq \Lip_\A(a)$, by a mild abuse of notations, we write $\targetsettunnel{\tau}{a}{l}$ for $\targetsettunnel{\tau'}{a}{l}$ where $\tau' = (\D,\Lip_\D,\pi_\A,\pi_\B)$.

Moreover, we denote $(\D,\Lip_\D,\pi_\B,\pi_\A,\varkappa,\varsigma)$ as $\tau^{-1}$. 

Now, by \cite[Corollary 4.5]{Latremoliere13b},\cite[Proposition 2.12]{Latremoliere14}, if $a, a' \in \dom{\Lip_\A}$ and $l\geq \max\{\Lip_\A(a),\Lip_\A(a')\}$, and if $b\in\targetsettunnel{a}{\tau}{l}$ and $b' \in \targetsettunnel{\tau}{a'}{l}$ then:
\begin{equation*}
  \norm{b - b'}{\B} \leq \norm{a - a'}{\A} + 2 l \tunnelextent{\tau} \text{.}
\end{equation*}

We now prove the coincidence property for our covariant propinquity.

\begin{theorem}
  Assume Hypothesis (\ref{hyp}). If $(\A,\Lip_\A,G,\delta_G,\alpha)$ and $(\B,\Lip_\B,H,\delta_H,\allowbreak \beta)$ in $\mathcal{C}$ then:
   \begin{equation*}
     \covpropinquity{\mathcal{T}}((\A,\Lip_\A,G,\delta_G,\alpha),(\B,\Lip_\B,H,\delta_H,\beta)) = 0
   \end{equation*}
if and only if there exists a full quantum isometry $\pi : (\A,\Lip_\A)\rightarrow(\B,\Lip_\B)$ and an isometric isomorphism of monoids $\varsigma: G\rightarrow H$ such that:
   \begin{equation*}
     \forall g \in G \quad \varphi\circ\alpha^g = \beta^{\varsigma(g)}\circ\varphi \text{.}
   \end{equation*}
i.e. $(\A,\Lip_\A,G,\delta_G,\alpha)$ and $(\B,\Lip_\B,H,\delta_H,\beta)$ are isomorphic as Lipschitz dynamical systems.
\end{theorem}

\begin{proof}
  Write $\mathds{A} = (\A,\Lip_\A,G,\delta_G,\alpha)$ and $\mathds{B} = (\B,\Lip_\B,H,\delta_H,\beta)$.

  Assume that there exists a full quantum isometry $\pi : (\A,\Lip_\A)\rightarrow(\B,\Lip_\B)$ and a proper monoid isometric isomorphism $\varsigma: G_1 \rightarrow G_2$ such that for all $g\in G_1$ we have $\pi\circ\alpha^g = \beta^{\varsigma(g)}\circ\pi$. As $\pi$ is a full quantum isometry, $(\A,\Lip_\A,\mathrm{id}_\A,\pi)$ is a tunnel (where $\mathrm{id}_\A$ is the identity of $\A$) of extent $0$. Let $\tau = (\A,\Lip_\A,\mathrm{id}_\A,\varsigma,\varsigma^{-1})$ --- which is, for all $\varepsilon > 0$, a $\varepsilon$-covariant tunnel from $\mathds{A}$ to $\mathds{B}$, and by definition, an element of $\mathcal{T}$. The extent of $\tau$ is of course $0$. We note that since $\pi$ is equivariant:
\begin{equation*}
\sup\left\{ \Kantorovich{\Lip_\A}(\varphi\circ\alpha^g, \varphi\circ\pi\circ\beta^{\varsigma(g)} ) : g \in G_1, \varphi \in \StateSpace(\A) \right\} = 0
\end{equation*}
and similarly exchanging $\A$ for $\B$, so $\tunnelreach{\tau}{\varepsilon} = 0$ for all $\varepsilon > 0$. Hence $\tunnelmagnitude{\tau}{\varepsilon} = 0 < \varepsilon$ for all $\varepsilon > 0$. Hence $\covpropinquity{\mathcal{T}}(\mathds{A},\mathds{B}) = 0$.

Conversely, assume that $\covpropinquity{\mathcal{T}}(\mathds{A},\mathds{B}) = 0$. For each $n\in\N$, there exists an $(n+2)$-covariant tunnel $\tau_n = (\D_n,\Lip_n,\pi_n,\rho_n,\varsigma_n,\varkappa_n)$ in $\mathcal{T}$ from $\mathds{A}$ to $\mathds{B}$ with $\tunnelmagnitude{\tau_n}{\frac{1}{n+2}} \leq \frac{1}{n+2}$. As the covariant propinquity dominates the propinquity (when $\covpropinquity{\mathcal{T}}\leq\frac{\sqrt{2}}{2}$) by Proposition (\ref{pseudo-prop}), we can use our work on the coincidence property of the propinquity from \cite[Theorem 4.16]{Latremoliere13b} and \cite[Theorem 5.13]{Latremoliere13}. Therefore, there exists a strictly increasing function $f : \N \rightarrow \N$ and a full quantum isometry $\pi: (\A,\Lip_\A)\rightarrow(\B,\Lip_\B)$ such that, for all $a\in\dom{\Lip_\A}$ and $l\geq\Lip_\A(a)$, the sequence $\left(\targetsettunnel{\tau_{f(n)}}{a}{l}\right)_{n\in\N}$ converges to $\{\pi(a)\}$ for $\Haus{\|\cdot\|_\B}$, and for all $b \in \dom{\Lip_\B}$ and $l\geq\Lip_\B(b)$, the sequence $\left(\targetsettunnel{\tau_{f(n)}^{-1}}{b}{l}\right)_{n\in\N}$ converges to $\{\pi^{-1}(b)\}$ for $\Haus{\|\cdot\|_\A}$. It remains to prove the equivariance property of $\pi$.

By Theorem (\ref{upsilon-metric-thm}), there exists a monoid isometric isomorphism $\varsigma : G \rightarrow H$ and a strictly increasing function $j : \N\rightarrow\N$ such that for all $g \in G$, the sequence $(\varsigma_{f(j(n))}(g))_{n\in\N}$ converges to $\varsigma(g)$, and for all $g\in H$, the sequence $(\varkappa_{f(j(n))}(g))_{n\in\N}$ converges to $\varsigma^{-1}(g)$. We write $\varkappa=\varsigma^{-1}$ and $k = f\circ j$.

Let $a\in \dom{\Lip_\A}$ and $h \in H$. Let $n\in\N$. Set $l = \Lip_\A(a)$. As $g\in G\mapsto \dil{\alpha^g}$ is locally bounded, there exists $\eta>0$ and $D \geq 0$ such that if $g\in G$ and $\delta_G(g,\varkappa(h)) < \eta$ then $\dil{\alpha^g} \leq D$. Since $\varkappa_{k(n)}(h)$ converges to $\varkappa(h)$ in $G$, there exists $N\in\N$ such that for all $n\geq N$, we have $\dil{\alpha^{\varkappa_{k(n)}(h)}} \leq D$.

Moreover, the sequence $(\varkappa_{k(n)}(h))_{n\in\N}$ is convergent, hence bounded, so there exists $N_1 \in \N$ such that for all $n\in\N$ we have $\varkappa_{k(n)}(h) \in G[N_1]$. Let $N_2 = \max\{ N, N_1 \}$.

For all $n\geq N_2$, we choose:
\begin{itemize}
\item $b_n \in \targetsettunnel{\tau_{k(n)}}{a}{l}$,
\item $c_n\in \targetsettunnel{\tau_{k(n)}}{\alpha^{\varkappa_{k(n)}(h)}(a)}{D l}$,
\item $a_n \in \targetsettunnel{\tau_{k(n)}}{\alpha^{\varkappa(h)}(a)}{D l}$.
\end{itemize}

We first note that (if $n\geq N_2$), using \cite[Proposition (4.4)]{Latremoliere13b}, \cite[Proposition 2.12]{Latremoliere14}:
\begin{align*}
  \norm{a_n - c_n}{\B} \leq \norm{\alpha^{\varkappa(h)}(a) -  \alpha^{\varkappa_{k(n)}(h)}(a)}{\A} + 2 D l \tunnelextent{\tau_{k(n)}}
\end{align*}
and since $\alpha$ is strongly continuous:
\begin{multline*}
  \limsup_{n\rightarrow\infty} \norm{a_n - c_n}{\B} \\
  \leq \limsup_{n\rightarrow\infty} \norm{\alpha^{\varkappa(h)}(a) -  \alpha^{\varkappa_{k(n)}(h)}(a)}{\A} + 2 l D \limsup_{n\rightarrow\infty} \tunnelextent{\tau_{k(n)}} = 0 \text{.}
\end{multline*}

Let $\psi \in \StateSpace(\B)$ and $n\geq N_2$. By definition of the magnitude of a covariant tunnel, there exists $\varphi_n \in \StateSpace(\A)$ such that:
\begin{equation*} 
  \Kantorovich{\Lip_{k(n)}}(\psi\circ\beta^g\circ\rho_{k(n)},\varphi_n\circ\alpha^{\varkappa_{k(n)}(g)}\circ\pi_{k(n)}) \leq \tunnelreach{\tau_{k(n)}}{\frac{1}{k(n)+2}}\leq \frac{1}{k(n)+2} \text{.}
\end{equation*}
Consequently, since $c_n\in\targetsettunnel{\tau_{k(n)}}{\alpha^{\varkappa_{k(n)}(h)}}{D l}$, we have:
\begin{equation*}
  \left|\psi(c_n) - \varphi_n(\alpha^{\varkappa_{k(n)}(h)}(a))\right| \leq D l \Kantorovich{\Lip_{k(n)}}(\psi,\varphi_n) < \frac{1}{k(n)+2}\text{.}
\end{equation*}
Similarly, since $b_n \in \targetsettunnel{\tau_{k(n)}}{a}{l}$, we have:
\begin{equation*}
  \left|\varphi_n(\alpha^{\varkappa_{k(n)}(h)}(a)) - \psi(\beta^h(b_n))\right| \leq \frac{l}{k(n)+2}\text{.}
\end{equation*}

We then compute for all $n\geq N_2$:
  \begin{multline*}
    \left|\psi(a_n - \beta^h(b_n))\right| \\
    \begin{split}
      &\leq \norm{a_n - c_n}{\B} + \left|\psi(c_n) - \varphi_n(\alpha^{\varkappa_{k(n)}(h)}(a))\right| + \left|\varphi_n(\alpha^{\varkappa_{k(n)}(h)}(a)) - \psi(\beta^h(b_n))\right| \\
      &\leq \norm{a_n - c_n}{\B} + \frac{(1 + D) l}{k(n)+2}  \\
      &\xrightarrow{n\rightarrow\infty} 0 \text{.}
    \end{split}
  \end{multline*}

Therefore:
\begin{equation*}
  \left| \psi(\pi(\alpha^{\varkappa(h)}(a)) - \beta^h(\pi(a))) \right|= \lim_{n\rightarrow\infty} \left|\psi(a_n - \beta^{h}(b_n))\right| = 0\text{.}
\end{equation*}

As $\psi \in \StateSpace(\B)$ is arbitrary, $\norm{\pi(\alpha^{\varkappa(h)}(a)) - \beta^h(\pi(a))}{\B} = 0$, i.e. $\pi(\alpha^{\varkappa(h)}(a)) = \beta^h(\pi(a))$. By linearity and continuity of $\pi$ and $\alpha^{\varkappa(h)}$, $\beta^h$, we conclude $\pi\circ\alpha^{\varkappa(h)} = \beta^{h}\circ\pi$ since $\dom{\Lip_\A}$ is a total subset of $\A$. As $h\in H$ was arbitrary, we conclude:
\begin{equation*}
  \forall h \in H \quad \pi\circ\alpha^{\varkappa(h)} = \beta^h\circ\pi \text{.}
\end{equation*}

Therefore:
\begin{equation*}
  \forall g \in G\quad \pi\circ\alpha^g = \pi\circ\alpha^{\varkappa(\varsigma(g))} = \beta^{\varsigma(g)} \text{.}
\end{equation*}
This concludes our proof.
\end{proof}

We thus have proven:
\begin{corollary}
  Assume Hypothesis (\ref{hyp}). The covariant propinquity $\covpropinquity{\mathcal{T}}$ is a metric up to equivariant full quantum isometry on the class $\mathcal{C}$ of Lipschitz dynamical $F$-systems.
\end{corollary}

\section{Convergence of Quantum Tori and Fuzzy Tori and their dual actions}

In this section, we prove that the dual actions on quantum tori form a continuous family for the $\covpropinquity{}$, and the dual actions on the fuzzy tori converge to the dual actions on quantum tori for $\covpropinquity{}$. We base this section on \cite{Latremoliere13c}.

Let $\Nbar = \N_\ast \cup \{\infty\}$, with $\N_\ast = \N\setminus\{0\}$, endowed with its usual topology:
\begin{equation*}
  \left\{ U \subseteq\Nbar : U\subseteq\N_\ast \text{ or }\Nbar\setminus U \text{ is finite}\right\}\text{.}
\end{equation*}

Fix $d \in \N\setminus\{0,1\}$ in this entire section. For any $k = (k_1,\ldots,k_d)\in\Nbar^d$, we write $\Z_k^d$ for $\bigslant{\Z^d}{\prod_{j=1}^d k_j \Z}$, with the convention $\infty\Z = \{ 0 \}$, and $\U_k^d$ for the Pontryagin dual of $\Z_k^d$ seen as a subgroup of the $d$-torus $\T^d = \{ (z_1,\ldots,z_d) \in \C^d : |z_1|=\ldots=|z_d| = 1\}$.

For this entire section, we fix a continuous length function $\ell$ on $\T^d$ and denote the metric it induces on $\T^d$ by $D_\ell$. We will in fact need to work with the metric $D$ induced by usual Hermitian $\norm{\cdot}{\C^d}$ as well. 

As a first observation:
\begin{equation*}
  \lim_{k\rightarrow \infty^d} \Haus{D}(\U_k^d, \T^d) = 0 \text{.}
\end{equation*}
By \cite[Proposition 3.4]{Latremoliere17c}, we have:
\begin{equation}\label{torus-cv-norm-eq}
    \lim_{k\rightarrow \infty^d} \Upsilon((\U_k^d,D),(\T^d,D)) = 0 \text{.}
\end{equation}
Moreover, following the proof of \cite[Proposition 3.4]{Latremoliere17c}, for $k\in\Nbar^d$, if we set $\varsigma_k(g) = g$ for all $g\in U_k^d$, and $\varkappa_k(h)$ is one of the closest element of $\U_k^d$ to $h \in \T^d$ for $\norm{\cdot}{\C^d}$ (it does not matter which such element we choose). Then:
\begin{equation*}
  (\varsigma_k,\varkappa_k) \in \UIso{\Upsilon((\U_k^d,D),(\T^d,D))}{(\U_k^d,D)}{(\T^d,D)}{2} \text{.}
\end{equation*}

Moreover, convergence in the sense of $\Haus{D}$ is the same as convergence for the Vietoris topology, and  then this is equivalent to convergence for $\Haus{D_\ell}$. Again by \cite[Proposition 3.4]{Latremoliere17c},  we conclude:
\begin{equation}\label{torus-cv-ell-eq}
    \lim_{k\rightarrow \infty^d} \Upsilon((\U_k^d,D_\ell),(\T^d,D_\ell)) = 0 \text{.}
\end{equation}

Let $A_d(\R)$ be the space of anti-symmetric $d\times d$-matrices over $\R$ metrized by the usual operator norm, and then set, for $(k_1,\ldots,k_d)\in\Nbar^d$:
\begin{equation*}
  \Xi_k^d = \left\{ (\theta_{n m})_{1\leq n,m \leq d} \in A_d(\R) \middle\vert \forall n,m \in \{1,\ldots,d\} \quad \mathrm{gcd}(k_n,k_m) \theta_{nm} \in \Z  \right\}
\end{equation*}
where $\mathrm{gcd}(n,m)$ is the greatest common divisor of $n$ and $m$ with the convention that $\mathrm{gcd}(\infty,n)=\mathrm{gcd}(n,\infty) = n$ for all $n\in\Nbar$.

Let $k \in \Nbar^d$. Any multiplier of $\Z_k^d$ is given by:
\begin{equation}\label{multiplier-eq}
  \sigma_{k,\theta} : (z,w) \in \Z_k^d \times \Z_k^d \mapsto \exp\left( 2i\pi \inner{\theta s(z)}{s(w)}{\C^d} \right)
\end{equation}
where $\theta\in\Xi_k^d$, the product $\inner{\cdot}{\cdot}{\C^d}$ is the usual inner-product on $\C^d$ and $s : \Z_k^d \rightarrow \Z^d$ is any section of the canonical surjection $\Z^d \rightarrow \Z_k^d$ --- by definition of $\Xi_k^d$, Expression (\ref{multiplier-eq}) is independent of the choice of $s$.

We denote the twisted convolution C*-algebra of $\Z^d_k$ by $\sigma_{k,\theta}$ for $\theta\in \Xi_k^d$ by $C^\ast(\Z_k^d,\theta)$. In the literature, $C^\ast(\Z^d,\theta)$ are known as quantum tori.

For all $k\in\Nbar^d$ and $\theta \in \U_k^d$, there is a natural strongly continuous ergodic action $\alpha_{k,\theta}$ of $\U_k^d$ on $C^\ast(\Z_k^d,\theta)$, called the dual action, which extends the action on the dense subspace $\ell^1(\Z_k^d)$ in $C^\ast(\Z_k^d,\theta)$ where we set, for $\xi \in \ell^1(\Z_k^d)$ and $\lambda=(\lambda_1,\ldots,\lambda_d)$:
\begin{equation*}
  \alpha_{k,\theta}\xi : z = (z_1,\ldots,z_d)\in\Z_k^d \mapsto \lambda^z \xi(z)
\end{equation*}
where we define $\lambda^k = (\lambda_1^{s(k_1)},\ldots,\lambda_d^{s(k_d)})$ with again $s : \Z_k^d\rightarrow\Z^d$ a section of the canonical surjection $\Z^d \rightarrow \Z_k^d$, whose choice does not influence the value of $\lambda^k$.

This action is instrumental in defining the noncommutative geometry of $C^\ast(\Z_k^d,\theta)$ via transport of structure. For our purpose, Rieffel proved in \cite{Rieffel98a} that for all $k\in\Nbar^d$, $\theta\in \Xi_k^d$, and $a\in C^\ast(\Z^d,\theta)$, if we set:
\begin{equation*}
  \Lip_{k,\theta}(a) = \sup\left\{ \frac{\norm{a - \alpha_{k,\theta}^\lambda(a)}{C^\ast(\Z_k^d,\theta)}}{\ell(\lambda)} \middle\vert \lambda \in \U_k^d \setminus \{1\} \right\},
\end{equation*}
then $\left( C^\ast(\Z_k^d,\theta), \Lip_{k,\theta}\right)$ is a Leibniz {\qcms}. In particular, the action $\alpha_{k,\theta}$ is an action by full quantum isometries.

We proved in \cite{Latremoliere13c} that:
\begin{equation*}
  \lim_{\substack{(k,\eta)\rightarrow(\infty_d,\theta)\\(q,\eta)\in\Omega}}\dpropinquity{}\left(\left(C^\ast(\Z_k^d,\eta),\Lip_{k,\eta}\right), \left(C^\ast(\Z^d,\theta),\Lip_{\infty^d,\theta}\right)\right) = 0
\end{equation*}
where $\Omega = \{ (k,\eta) \in \Nbar^d\times A(\R^d) : \eta\in\Xi_k^d \}$ is topolgized as a subset of the product $\Nbar^d\times A(\R^d)$. In this section, we work with the propinquity --- and the covariant propinquity --- for the class of all Leibniz tunnels over Leibniz {\qcms s}.

The computation of an upper bound on the propinquity between quantum tori and fuzzy tori in \cite{Latremoliere13c} relies on a particular choice of *-representations of these C*-algebras. A noteworthy property of this choice is that all the dual actions described above are implemented using a particular family of unitaries which does not depend on the choice of multiplier.

Fix $k\in\Nbar^d$ and $\theta\in \Xi_k^d$. We denote by $\ell^2(\Z_k^d)$ the Hilbert space of square summable, $\C$-valued families indexed by $\Z_k^d$. For all $n\in\Z_k^d$, we write:
\begin{equation*}
  e_k^n : m \in \Z_k^d \mapsto \begin{cases}
    1 \text{ if $n = m$,}\\
    0 \text{ otherwise.}
  \end{cases}
\end{equation*}
and we note that $(e_k^n)_{n\in\Z_k^d}$ is a Hilbert basis of $\ell^2(\Z_k^d)$. We also denote the C*-algebra of all bounded linear operators on $\ell^2(\Z_k^d)$ by $\B_k$. To ease our notations later on, we set $\B_{\infty,\ldots,\infty} = \B$ and $e^{\infty,\ldots,\infty}_n = e_n$ for all $n\in\Z^d$.

For each $n\in\Z^d_k$ and $\theta \in \Xi_k^d$, we define $U_{k,\theta}^n\in\B_k$ by linearity and continuity from the relations:
\begin{equation*}
  U_{k,\theta}^n e_k^m = \sigma_{k,\theta}(n,m) e_k^{m-n}\text{.}
\end{equation*}
Since $U_{k,\theta}^n  U_{k,\theta}^m = \sigma_{k,\theta} (n,m) U_{k,\theta}^{n+m}$ for all $n,m \in \Z_k^d$, by universality, there exists a unique *-representation $\rho_{k,\theta}$ of $C^\ast(\Z_k^d,\theta)$ on $\ell^2(\Z_k^d)$ such that $\rho_{k,\theta}(f) = \sum_{n\in\Z^d_k} f(n) U_{k,\theta}^n$ for all $f$ in the dense subspace $\ell^1(\Z_k^d)$ of $C^\ast(\Z_k^d,\theta)$.

On $\ell^2(\Z^d)$, for any $\omega\in\T^d$,$\lambda\in\U_k^d$, we define the unitary $u_{k,\omega,\lambda}\in\B_k$ by setting $u_{k,\omega,\lambda}(e_k^n) = \omega \lambda^{-n} e_n$ for all $n\in\Z_k^d$ where we multiply the $d$-tuples of $\T^d$ coordinatewise. We now compute, for all $\lambda\in\U_k^d$, $\omega\in\T^d$:
\begin{align*}
  u_{k,\omega,\lambda} U_{k,\theta}^n u_{k,\omega,\lambda}^\ast(e_m) 
  &= u_{k,\omega,\lambda} U_{k,\theta}^n \overline{\omega}\lambda^{m} e_k^m \\
  &= \overline{\omega} \lambda^{m} u_{k,\omega,\lambda} \sigma_{k,\theta}(n,m)e_k^{m-n} \\
  &= \overline{\omega} \lambda^m \sigma_{k,\theta}(n,m) \omega \lambda^{n-m} e_k^{m-n}\\
  &= \lambda^{n} \sigma_{k,\theta}(n,m) e_k^{m-n} \\
  &= \lambda^{n} U_{k,\sigma}^n e_k^m \text{.}
\end{align*}

Therefore, for all $\lambda\in\U_k^d$ and $\omega\in\T^d$, we have $\rho_{k,\sigma}(\alpha_{k,\sigma}^\lambda(a)) = \mathrm{Ad}u_{k,\omega,\lambda} (\rho_{k,\sigma}(a))$ for all $a\in C^\ast(\Z^d_k,\theta)$ --- as can be directly checked for $a\in\ell^1(\Z_k^d)$ and then holds by continuity over all of $C^\ast(\Z_k^d,\theta)$.

In order to compute an upper bound on the propinquity between any two of our twisted C*-algebras of quotients of products of cyclic groups, it proved useful in \cite{Latremoliere13c} to represent all these C*-algebras on the same Hilbert space $\ell^2(\Z^d)$.

Define:
\begin{equation*}
  I_k = \prod_{j=1}^d J_{k_j}
\end{equation*}
where for any $n\in\Nbar$:
\begin{equation*}
  J_n = \begin{cases}
    \left\{ \left\lfloor \frac{1-n}{2} \right\rfloor, \left\lfloor \frac{1-n}{2} \right\rfloor + 1, \ldots, \left\lfloor \frac{n-1}{2} \right\rfloor \right\} \text{ if $n \in \N_\ast$,}\\
    \Z \text{ otherwise.}
\end{cases}
\end{equation*}

Let $q : \Z^d\twoheadrightarrow\Z_k^d$ be the canonical surjection over, which we observe restricts to a bijection over $I_k$. Let $\mathcal{E}_m = \mathrm{span}\left\{e_{n + m} : n \in I_k \right\}$ for all $m\in k\Z^d$. For $m\in k\Z^d$, we define $y_m$ by linearity and continuity from:
\begin{equation*}
  y_m : e^n \mapsto\begin{cases}
    e_k^{q(n - m)} \text{ if $n \in m + I_k$,}\\
    0 \text{ otherwise.}
  \end{cases}
\end{equation*} 

We then extend $\rho_{k,\theta}$ to a representation of $C^\ast(\Z_k^d,\theta)$ on $\B$ by setting:
\begin{equation*}
  \pi_{k,\theta} (a)  = \sum_{m \in k\Z^d} y_m^\ast \rho_{k,\theta}(a) y_m\text{.}
\end{equation*}
Let $u_\lambda$ be our simplified notation for $u_{\infty^d,1,\lambda}$. We now compute for all $m\in k\Z^d$ and $n\in m + I_k$:
\begin{equation*}
   y_m u_\lambda(e^n) = \lambda^{-n}e_k^{q(n - m)} = \lambda^{m} \lambda^{n - m} y_m(e^n) = u_{\lambda^m,\lambda} y_m(e^n) \text{.}
\end{equation*}
Of course, $y_m u_\lambda(e^n) = 0 = u_{\lambda^m,\lambda} y_m(e^n)$ for all $n\not\in m+I_k$.

Hence for all $\lambda\in\U_k^d$, we compute:
\begin{align*}
  \mathrm{Ad}u_\lambda\circ\pi_{k,\theta}(\cdot) &= \sum u_\lambda y_m^\ast \rho_{k,\theta}(\cdot) y_m u_\lambda^\ast \\
  &= \sum y_m^\ast u_{k,\lambda^m,\lambda} \rho_{k,\theta}(\cdot) u_{k,\lambda^m,\lambda}^\ast y_m \\
  &= \sum y_n^\ast \rho\circ\alpha^\lambda_{k,\theta}(\cdot) y_m \text{.}
\end{align*}
We thus conclude that $\pi_{k,\theta}\circ\alpha^\lambda = \mathrm{Ad}u_\lambda\circ\pi_{k,\theta}$ for all $\lambda\in\U_k^d$.

In \cite{Latremoliere13c}, we computed an upper bound on the propinquity between quantum tori and fuzzy tori (or other quantum tori) from the construction of a bridge, in the sense of \cite{Latremoliere13}, which we now recall:
\begin{definition}[\cite{Latremoliere13}]\label{bridge-def}
  A \emph{bridge} $(\D,x,\pi_\A,\pi_\B)$ from a unital C*-algebra $\A$ to a unital C*-algebra $\B$ is given by a unital C*-algebra $\D$, two unital *-monomorphisms $\pi_\A : \A\rightarrow\D$ and $\pi_\B : \B\rightarrow\D$, and an element $x\in \D$ with the property that if $y=\unit_\D - x$ then:
  \begin{equation*}
    \StateSpace_1(\D|x) = \left\{ \varphi \in \StateSpace(\D) \middle\vert \varphi(y^\ast y) = \varphi(y y^\ast) = 0 \right\} \not= \emptyset \text{.}
  \end{equation*}
\end{definition}

A bridge provides a mean to measure how far two {\qcms s} are by using the notion of length.
\begin{definition}[\cite{Latremoliere13}]
  Let $(\A_1,\Lip_1)$ and $(\A_2,\Lip_2)$ be two {\qcms s}. The \emph{length} $\bridgelength{\gamma}{\Lip_1,\Lip_2}$ of a bridge $\gamma = (\D,x,\pi_1,\pi_2)$ from $\A_1$ to $\A_2$ is the maximum of its \emph{height}, defined by:
  \begin{equation*}
    \max_{j\in\{1,2\}}\left\{ \Haus{\Kantorovich{\Lip_j}}\left(\StateSpace(\A_j),\{\varphi\circ\pi_j : \varphi \in \StateSpace_1(\D|x)\right) \right\}
  \end{equation*}
  and its \emph{reach}, defined by:
  \begin{equation*}
    \max_{\{j,k\}=\{1,2\}} \sup_{\substack{a_j \in \sa{\A_j}\\ \Lip_j(a_j)\leq 1}} \inf_{\substack{a_k \in \sa{\A_k} \\ \Lip_k(a_k) \leq 1}} \bridgenorm{\gamma}{a_1,a_2}
  \end{equation*}
  where $\bridgenorm{\gamma}{a,b} = \norm{\pi_1(a) x - x \pi_2(b)}{\D}$ for all $a\in\A_1$ and $b\in \A_2$.
\end{definition}
We refer to \cite{Latremoliere13} for the observation that the length of a bridge is always finite and that bridges always exist, as well as their use in computing estimates on the propinquity. We will now discuss how to generalize this notion to prove convergence for the covariant propinquity.

Let $\varepsilon > 0$ and recall from above our choice of $\B = \ell^2(\Z^d)$ and the representations $\pi_{k,\theta}$ for all $(k,\theta)\in\Omega$. Fix $\theta \in A(\R^d)$. 

By \cite{Latremoliere13c}, there exists an open neighborhood $V$ of $(\infty_d,\theta)$ in $\Omega$ and a finite rank operator $x\in \B$, diagonal in the basis $(e^n)_{n\in\N}$, such that for all $(k,\eta) \in V$, the quadruple $\gamma_{k,\eta} = (\B,x_{k,\eta},\pi_{\infty^d,\theta},\pi_{k,\eta})$ is a bridge $C^\ast(\Z^d,\theta)$ to $C^\ast(\Z_k^d,\eta)$ where:
\begin{itemize}
  \item $\bridgelength{\gamma_{k,\eta}}{\Lip_\A,\Lip_\B} \leq \frac{\varepsilon}{2}$,
  \item if $a\in C^\ast(\Z^d_k,\theta)$ with $\Lip_\theta(a) \leq 1$ then $\norm{[x_{k,\eta},\pi_{\infty^d,\theta}(a)]}{\B} < \varepsilon$.
\end{itemize}
Note in particular that for all $\lambda\in\T^d$, the operators $x$ and $u_\lambda$ commute. It will be important that $x$ is chosen uniformly over $V$, and this can be seen in \cite[Claim 5.2.7]{Latremoliere13c} of the proof of \cite[Theorem 5.2.4]{Latremoliere13c}.

Since $x$ is finite rank, diagonal in the basis $(e^n)_{n\in\Z^d}$, there exists $N\in\N$ such that if $n=(n_1,\ldots,n_d) \in \Z^d$ and $|n_j| > N$ for some $j\in\{1,\ldots,d\}$ then $x(e^n) = 0$.

The function which associates to a {\qcms} the diameter of its state space for the {\MongeKant} is continuous with respect to the propinquity, and the class $\left\{ \left(C^\ast(\Z_k^d),\theta),\Lip_{k,\theta}\right) : (k,\theta)\in \Omega \right\}$ is compact for $\dpropinquity{}$, so there exists an upper bound $R > 0$ to all the diameters of the state space of all the spaces in this class.

Now, we can choose $M \in \N$ with $M \geq 2 N$ such that if $n\geq M$ then $\sqrt{d} \left|\sin\left(\frac{N \pi}{n}\right)\right| < \frac{\varepsilon}{4 R}$. Then, for all $k = (k_1,\ldots,k_d) \in \Nbar^d$ with $\min\{k_j:j\in\{1,\ldots,d\}\} \geq M$, we compute for all $m=(m_1,\ldots,m_d)\in\N^d$ with $\max_{j\in\{1,\ldots,d\}} |m_j| \leq N$: 
\begin{align*}
  \norm{\lambda^m - \varkappa_k(\lambda)^m}{\C^d} 
  &= \norm{1-\left(\lambda^{-1}\varkappa_k(\lambda)\right)^m}{\C^d} \\
  &\leq \sqrt{d} \max\left\{ \left|1 - \exp\left(\frac{2i N \pi}{k_j}\right)\right| : j \in \{1,\ldots,d\} \right\}\\
  &\leq \sqrt{d} \max\left\{ 2\left| \sin\left( \frac{N \pi}{k_j} \right)  \right| : j \in \{1,\ldots,d\} \right\}\\
  &\leq \frac{\varepsilon}{2 R} \text{.}
\end{align*}

Moreover, let $\mu \in \StateSpace(C^\ast(\Z^d,\theta))$. We then note that for all $a\in\sa{C^\ast(\Z^d,\theta)}$, we have:
\begin{equation*}
\forall\varphi\in\StateSpace(C^\ast(\Z^d,\theta)) \quad |\varphi(a - \mu(a))| \leq |\varphi(a)-\mu(a)| \leq R \Lip_{\infty^d,\theta}(a)
\end{equation*}
so:
\begin{equation*}
  \norm{a-\mu(a)}{C^\ast(\Z^d,\theta)} \leq R \Lip_{\infty^d,\theta}(a) \text{.}
\end{equation*}

Let now $V_1 = V \cap \left( \left\{ (k_1,\ldots,k_d) \in \Nbar^d : \min_{j\in\{1,\ldots,d\}} |k_j| \geq M \right\} \times A(\R^d) \right)$. Note that $V_1$ is a neighborhood of $(\infty^d, \theta)$ in $\Omega$.

In the rest of this section, fix $(k,\eta) \in V_1$. We write $\pi_\eta$ for $\pi_{k,\eta}$, and $\pi_\theta$ for $\pi_{\infty^d,\theta}$. We also simply write $\alpha_\theta$ and $\alpha_\eta$ for the actions $\alpha_{\infty^d,\theta}$ and $\alpha_{k,\eta}$ respectively. We write $\gamma$ for $\gamma_{k,\eta}$.

Now, if $a\in \sa{C^\ast(\Z^d,\theta)}$ and $b \in \sa{C^\ast(\Z_k^d,\eta)}$, and if $\lambda \in \T^d$ and $z = \varkappa_k(\lambda) \in \U_k^d$, then:
\begin{align*}
  \bridgenorm{}{\alpha_\theta^\lambda(a),\alpha_\eta^z(b)}
  &= \norm{\pi_\theta(\alpha_\theta(a)) x - x\pi_\eta(\alpha_\eta^z(b))}{\B}\\
  &= \norm{u_\lambda \pi_\theta(a) u_\lambda^\ast x - x u_z \pi_\eta(b) u_z^\ast}{\B}\\
  &= \norm{u_\lambda \pi_\theta(a) x u_\lambda^\ast  - u_z x \pi_\eta(b) u_z^\ast}{\B}\\  
  &\leq \norm{u_\lambda \pi_\theta(a) x u_\lambda^\ast  - u_z \pi_\theta(a) x u_z^\ast}{\B} + \norm{u_z \pi_\theta(a) x u_z^\ast  - u_z x \pi_\eta(b) u_z^\ast}{\B}\\
  &= \norm{u_\lambda \pi_\theta(a-\mu(a)\unit) x u_\lambda^\ast  - u_\lambda \pi_\sigma(a-\mu(a)\unit) x u_z^\ast}{\B} \\
  &\quad + \norm{u_\lambda \pi_\theta(a-\mu(a)\unit) x u_z^\ast  - u_z \pi_\sigma(a-\mu(a)\unit) x u_z^\ast}{\B} \\
  &\quad + \norm{\pi_\theta(a) x  - x \pi_\eta(b)}{\B}\\
  &\leq  \norm{\pi_\theta(a-\mu(a)\unit) x \left(u_\lambda^\ast  - u_z\right)}{\B} \\
  &\quad +  \norm{\left(u_\lambda^\ast  - u_z\right) \pi_\theta(a-\mu(a)\unit) x}{\B} + \bridgenorm{\gamma}{a,b} \\
  &\leq 2\norm{(u_z-u_\lambda)x}{\B} \norm{a-\mu(a)\unit}{\A} + \bridgenorm{\gamma}{a,b}\\
  &\leq 2 R \max_{\substack{m=(m_1,\ldots,m_d)\\m_1\leq N,\ldots,m_d\leq N}}\norm{z^m - \lambda^m}{\C^d}  + \bridgenorm{\gamma}{a,b} \\
  &\leq  \varepsilon \text{.}
\end{align*}

We now investigate how our computation can be used to prove a convergence result for the covariant propinquity. We abstract some of the observations we have made above. To keep our notations reasonable for this particular example, we work within the context of actions from compact groups.

\begin{definition}\label{covariant-bridge-def}
    Let $\mathds{A}_1 = (\A_1,\Lip_1,G_1,\delta_1,\alpha_1)$ and $\mathds{A}_2 = (\A_2,\Lip_2,G_2,\delta_2,\alpha_2)$ be two Lipschitz dynamical systems with $\alpha_1$, $\alpha_2$ actions by unital *-endomorphisms. A \emph{covariant bridge} $\gamma  = (\D,x,\pi_\A,\pi_\B,\varsigma_1,\varsigma_2)$ is given by:
  \begin{enumerate}
    \item a bridge $(\D,x,\pi_\A,\pi_\B)$,
    \item a map $\varsigma_1 : G_1 \rightarrow G_2$ and a map $\varsigma_2  : G_2\rightarrow G_1$, both mapping the identity element to the identity element.
  \end{enumerate}
\end{definition}

\begin{definition}
  Let $\mathds{A}_1 = (\A_1,\Lip_1,G_1,\delta_1,\alpha_1)$ and $\mathds{A}_2 = (\A_2,\Lip_2,G_2,\delta_2,\alpha_2)$ be two Lipschitz dynamical systems with $\alpha_1$,$\alpha_2$ actions by unital *-endomorphisms. Let $\gamma = (\D,x,\pi_1,\pi_2,\varsigma_1,\varsigma_2)$ be a covariant bridge from $\mathds{A}_1$ to $\mathds{A}_2$. The \emph{reach} $\bridgereach{\gamma}{\Lip_\A,\Lip_\B}$ of $\gamma$ is:
  \begin{equation*}
    \max_{ \{j,k\}=\{1,2\} } \; \sup_{\substack{ a\in\dom{\Lip_j} \\ \Lip_j(a) \leq 1}} \; \inf_{\substack{b\in\dom{\Lip_k} \\ \Lip_k(b)\leq 1}} \; \sup_{g\in G_j}  \bridgenorm{\gamma}{\alpha_j^g(a),\alpha_k^{\varsigma_j(g)}(b)} \text{.} 
  \end{equation*}

  The \emph{length} $\bridgelength{\gamma}{\Lip_1,\Lip_2}$ of $\gamma$ is the maximum of its reach and the height of the bridge $(\D,x,\pi_1,\pi_2)$.
\end{definition}

\begin{proposition}\label{cobridge-prop}
  Let $F$ be a permissible function. Let $\mathds{A} = (\A,\Lip_\A,G,\delta_G,\alpha)$ and $\mathds{B} = (\B,\Lip_\B,H,\delta_H,\beta)$ be two Lipschitz dynamical $F$-systems, with $\alpha$ and $\beta$ being actions via unital *-endomorphisms. Let $\gamma = (\D,x,\pi_\A,\pi_\B,\varsigma,\varkappa)$ be a covariant bridge from $\mathds{A}$ to $\mathds{B}$.
  If $\varepsilon > 0$ is chosen so that:
  \begin{itemize}
  \item $\varepsilon \geq \bridgelength{\gamma}{\Lip_\A,\Lip_\B}$,
  \item $(\varsigma,\varkappa) \in \UIso{\varepsilon}{(G,\delta_G)}{(H,\delta_H)}{\frac{1}{\varepsilon}}$,
  \end{itemize}
  and if for all $a\in\dom{\Lip_\A}$ and $b\in\dom{\Lip_\B}$, we set: 
  \begin{multline*}
    \Lip(a,b) = \sup\left\{ \Lip_\A(a),\Lip_\B(b), \frac{1}{\varepsilon}\bridgenorm{\gamma}{\alpha^g(a),\beta^{\varsigma(g)}(b)}, \right. \\
      \left. \frac{1}{\varepsilon}\bridgenorm{\gamma}{\alpha^{\varkappa(h)}(a),\beta^h(b)} : g \in G, h \in H  \right\}
  \end{multline*}
  then $\tau = (\A\oplus\B,\Lip,\rho_\A,\rho_\B,\varsigma,\varkappa)$ is a $\frac{1}{4\varepsilon}$-covariant $F$-tunnel such that $\tunnelmagnitude{\tau}{\frac{1}{4\varepsilon}} \leq 4\varepsilon$.

  In particular:
  \begin{equation*}
    \covpropinquity{F}(\mathds{A},\mathds{B}) \leq 4 \varepsilon \text{.}
  \end{equation*}
\end{proposition}

\begin{proof}
  By construction, and since $\alpha$ and $\beta$ are actions by unital *-endomorphisms, the seminorm $\Lip$ if $F$-quasi-Leibniz. Moreover, if $(a,b) \in \dom{\Lip_\A}\oplus\dom{\Lip_\B}$ then $(a,b) \in \dom{\Lip}$, and since $\dom{\Lip_\A}\oplus\dom{\Lip_\A}$ is dense in $\sa{\A\oplus\B}$, we conclude that $\dom{\Lip}$ is dense in $\sa{\A\oplus\B}$. Moreover, $\Lip$ is lower-semi-continuous since $\Lip_\A$ and $\Lip_\B$ are lower semi-continuous and $\bridgenorm{\gamma}{\cdot}$ is continuous, as well as the maps $\alpha^g$ and $\beta^h$ for $g\in G$ and $h \in H$.

  We also note that if $\Lip(a,b) = 0$ then $a = t\unit_\A$ and $b = s\unit_\B$ for some $t,s \in \R$. Now $\bridgenorm{\gamma}{a,b} = 0$ as well, so $\norm{(t-s)x}{\D} = 0$ and thus $t = s$. Hence $(a,b) \in \R\unit_{\A\oplus\B}$. Of course $\Lip(\unit_{\A\oplus\B}) = 0$.

  We also note that $\Lip$ dominates the Lip-norm defined in \cite[Theorem 6.3]{Latremoliere13}, and therefore by \cite[Theorem 1.10]{Rieffel98a}, $\Lip$ is a Lip-norm. Hence $\Lip$ is an $F$-quasi-Leibniz L-seminorm.

  If $a\in\sa{\A}$ then $\Lip_\A(a)\leq\Lip(a,b)$ for all $b\in \sa{\B}$. If moreover $\Lip_\A(a)<\infty$, by definition of the reach of $\gamma$, there exists $b \in \B$ with $\Lip_\B(b) \leq \Lip_\A(a)$ and $\bridgenorm{\gamma}{\alpha^g(a),\beta^{\varsigma(g)}(b)} \leq \bridgereach{\gamma}{\Lip_\A,\Lip_\B} \leq\varepsilon$ for all $g \in G$. Thus $\Lip(a,b) \leq \Lip_\A(a)$. The reasoning is symmetric with respect to $\A$ and $\B$ and thus we conclude that $(\A\oplus\B,\Lip,\rho_\A,\rho_\B)$ is a tunnel. By assumption, $\tau$ is a $\varepsilon$-covariant $F$-tunnel.

  Let now $\varphi \in \StateSpace(\A)$. By definition of the length of a bridge, there exists $\psi \in \StateSpace(\D)$ with $\Kantorovich{\Lip}(\varphi,\psi\circ\pi_\A) \leq \bridgelength{\gamma}{\Lip_\A,\Lip_\B}$ and $\psi(x d) = \psi(d x) = \psi(d)$ for all $d\in \D$. We write $\psi_\A = \psi\circ\pi_\A$ and $\psi_\B = \psi\circ\pi_\B$. We also identify all states of $\A$ (resp. $\B$) with states over $\A\oplus\B$ by setting $\mu(a,b) = \mu(a)$ for all $(a,b) \in \A\oplus\B$ and $\mu\in\StateSpace(\A)$.

  Let $g \in G$. Let $(a,b) \in \A\oplus\B$ with $\Lip(a,b) \leq 1$. We then compute:
  \begin{align*}
    \left|\psi_\A\circ\alpha^g (a,b) - \psi_\B\circ\beta^{\varsigma(g)}(a,b) \right| &=  \left|\psi(\pi_\A(\alpha^g (a))) - \psi(\pi_\B(\beta^{\varsigma(g)}(b))) \right|\\
    &=\left|\psi(\pi_\A(\alpha^g (a))x) - \psi(x\pi_\B(\beta^{\varsigma(g)}(b))) \right|\\
    &=\left|\psi(\pi_\A(\alpha^g (a))x - x\pi_\B(\beta^{\varsigma(g)}(b))) \right|\\
    &\leq \bridgenorm{\gamma}{\alpha^g(a),\beta^{\varsigma(g)}(b)}\\
    &\leq \bridgelength{\gamma}{\Lip_\A,\Lip_\B} \text{.}
\end{align*}

Hence $\Kantorovich{\Lip}(\psi_\A\circ\alpha^g,\psi_\B\circ\beta^{\varsigma(g)}) \leq \bridgelength{\gamma}{\Lip_\A,\Lip_\B}$. We conclude that $\Kantorovich{\Lip}(\varphi\circ\alpha^g,\psi_\B\circ\beta^g) \leq 2 \bridgelength{\gamma}{\Lip_\A,\Lip_\B}$. Hence $\tunnelreach{\tau}{r} \leq 2\bridgelength{\gamma}{\Lip_\A,\Lip_\B}$ for all $r > 0$ --- since we actually did our computation for all $g \in G$.

The extent of $\tau$ is easy to compute: the above computation proves the length of the tunnel $(\A\oplus\B,\Lip,\rho_\A,\rho_\B)$ is no more than $2\bridgelength{\gamma}{\Lip_\A,\Lip_\B}$. Hence the extent of $\tau$ is also no more than $4\bridgelength{\gamma}{\Lip_\A,\Lip_\B}$ by \cite[Proposition 2.12]{Latremoliere14}.
Hence $\tunnelmagnitude{\tau}{\frac{1}{4\varepsilon}} \leq 4\varepsilon$. This concludes our proof.
\end{proof}

\begin{remark}[The covariant quantum propinquity]
  Using covariant bridges, we can construct a covariant quantum propinquity in the model of \cite{Latremoliere13}, using treks of covariant bridges using a natural definition (simply chaining the local isometric isomorphisms). Proposition (\ref{cobridge-prop}) would then imply that the resulting pseudo-metric dominates the covariant propinquity, and thus it is a metric.
\end{remark}

We thus have proven:

\begin{theorem}
  \begin{equation*}
    \lim_{\substack{\eta\rightarrow\theta \\ k\rightarrow \infty^d }} \covpropinquity{}\left( \left(C^\ast(\Z^d_k,\eta),\Lip_{k,\eta},\U_k^d,D_\ell,\alpha_{k,\eta}\right), \left(C^\ast(\Z^d,\theta),\Lip_{\infty^d,\theta},\T^d,D_\ell,\alpha_{\infty^d,\theta} \right) \right) = 0 \text{.}
  \end{equation*}
\end{theorem}

\begin{proof}
  Our theorem follows from Proposition (\ref{cobridge-prop}) applied to the construction outlined in this section.
\end{proof}

\bibliographystyle{amsplain}
\bibliography{../thesis}
\vfill

\end{document}